\documentclass[1p]{elsarticle}

\usepackage{lineno}
\usepackage[colorlinks,linkcolor={blue}]{hyperref}
\modulolinenumbers[5]

\usepackage{amsfonts}
\usepackage{graphicx}
\usepackage{amsmath}
\usepackage{graphicx}
\usepackage{amssymb, nicefrac}
\usepackage{enumitem}
\numberwithin{equation}{section}

\usepackage{mathtools}

\usepackage{xcolor}
\definecolor{lightblue}{rgb}{0.8,0.8,1}  

\newtheorem{theorem}{Theorem}

\newtheorem{definition}[theorem]{Definition}
\newtheorem{example}[theorem]{Example}

\newtheorem{lemma}[theorem]{Lemma}

\newtheorem{proposition}[theorem]{Proposition}

\newtheorem{remark}[theorem]{Remark}

\newenvironment{proof}[1][Proof]{\textbf{#1.} }{\ \rule{0.5em}{0.5em}}

\newcommand{\comment}[1]{}
\newcommand{\cover}[1]{\stackrel{#1}{\Longrightarrow}}

\newcommand{\mymarginpar}[1]{}
\newcommand{\correction}[1]{#1}

\makeatletter
\usepackage{tikz}
\newcommand*\circled[2][1.6]{\tikz[baseline=(char.base)]{
    \node[shape=circle, draw, inner sep=1pt,
        minimum height={\f@size*#1},] (char) {\vphantom{WAH1g}#2};}}

\makeatother

\begin{document}

\title{Characterising blenders via 
  covering relations and cone conditions}

\author{Maciej
J. Capi{\'n}ski}\ead{maciej.capinski@agh.edu.pl}\address{Faculty of Applied Mathematics, AGH University of Krak\'ow, al. Mickiewicza 30, 30-059 Krak\'ow, Poland}
\author{Bernd Krauskopf}\ead{b.krauskopf@auckland.ac.nz} \quad \author{Hinke
M. Osinga} \ead{h.m.osinga@auckland.ac.nz}\address{ Department of Mathematics, University of Auckland, Private Bag 92019, Auckland 1142, New Zealand}
\author{Piotr
Zgliczy{\'n}ski}\ead{umzglicz@cyf-kr.edu.pl}\address{Institute of Computer Science, Jagiellonian University, ul. prof. Stanis{\l}awa {\L}ojasiewicza 6, 30-348 Krak{\'o}w,  Poland}

\mymarginpar{\vspace{8cm}some sentences have been shortened or removed from the abstract.}
\begin{abstract}
We present a characterisation of a blender based on the topological alignment of certain sets in phase space in combination with \correction{cone conditions}. Importantly, the required conditions can be verified by checking properties of a single iterate of the diffeomorphism, which is achieved by finding finite series of sets that form suitable sequences of alignments. This characterisation is applicable \correction{in arbitrary dimension}. \correction{Moreover, the approach naturally extends to establishing $C^1$-persistent heterodimensional cycles.} Our setup is flexible and allows for a rigorous, computer-assisted validation based on interval arithmetic. 
\end{abstract}



\maketitle

\section{Introduction}
\hypertarget{target:first-paragraph}{}
\mymarginpar{\hyperlink{issue:2}{Link back to \ref{first-paragraph}.}}
\correction{The notion of a \emph{blender} was introduced by Bonatti and D\'iaz~\cite{MR1381990} in order to construct a large class of examples of $C^1$-persistent, nonhyperbolic, transitive diffeomorphisms. The objective of this paper is to provide a tool that can be used to validate the existence of blenders.}

\correction{In this paper we adopt the following definition.
\mymarginpar{\hyperlink{issue:2}{Link back to \ref{issue:2}.}}\hypertarget{target:2}{}
\hypertarget{target:blender}{}
\mymarginpar{\hyperlink{issue:blender}{Link back to \ref{issue:blender}.}}
\hypertarget{target:blender2}{}\mymarginpar{\hyperlink{issue:mother}{Link back to \ref{issue:blender2}.}}
\begin{definition}[$k$-blender]\cite{bonatti2012AMS}
\label{def:blender}
Let $\Lambda \subset \mathbb{R}^n$ be an invariant transitive hyperbolic set with $d_s$-dimensional stable fibres. For $k \in \mathbb{N}$ with $k > d_s$, we say that $\Lambda$ is a \emph{$k$-blender}
if there are a $C^1$-neighbourhood $\mathcal{U}$ of $f$ and a $C^1$-open set $\mathcal{S}$ of embeddings of $(n-k)$-dimensional disks into $\mathbb{R}^n$, such that for every diffeomorphism $g\in \mathcal{U}$, for every continuation $\Lambda_g$ of $\Lambda$, and for every $S \in\mathcal{S}$,
\begin{equation}
  S \cap W^s_{\mathrm{\correction{loc}}}(\Lambda_g) \neq \emptyset. \label{eq:key-objective}
\end{equation}
\end{definition}
Intuitively, a blender is a hyperbolic set with a stable (or unstable) manifold that `behaves' as a higher-dimensional set than is expected from the underlying splitting.}

\mymarginpar{\hyperlink{issue:objectives}{Link back to \ref{issue:objectives}.}}\hypertarget{target:objectives}{}
\correction{Our main aim is to provide a flexible framework for validating \eqref{eq:key-objective}. We remark that there are a number of different but related definitions of blenders (examples include~\cite{Barrientos, MR1381990, bonatti2012AMS, bonatti2012Nonlin, bonatti2005}) in a variety of different contexts.
We believe that the approach presented in this paper is also applicable to other types of blenders. More specifically, in our construction we establish the existence of a compact invariant set, within a prescribed domain, that satisfies property \eqref{eq:key-objective}. Our method can then be applied to prove the existence of a specific type of blender if this invariant set satisfies additional conditions required for such a blender (or it is a subset of a larger hyperbolic set satisfying these conditions).}

\mymarginpar{\hyperlink{Plykin}{Link back to \ref{Plykin}.}}\hypertarget{target:Plykin}{}
\correction{Our aim is that our approach and the associated tools be flexible. It is worth noting that some types of blenders, for instance derived from Anosov or Plykin attractors, require very specific underlying dynamics, involving global constructions. In contrast, the approach presented in this paper is significantly more flexible, relying only on much weaker local or semi-local conditions.}

Our method is based on finding a family of sets in phase space with the property that the image of each such set intersects another set in a `topologically good way'. This is formalised by the notion of a \emph{covering relation}, as introduced in~\cite{MR2060531} and defined in Section~\ref{sec:prel}. While the number of sets in the family may be large, the proof of existence requires assumptions on their images under only one single iteration, which means verification can be done in a local, far more efficient manner, with better control over numerical accuracy.

More specifically, we consider throughout a diffeomorphism
\begin{displaymath}
  f : \mathbb{R}^n \to \mathbb{R}^n,
\end{displaymath}
\mymarginpar{\hyperlink{coordinates-removed}{Link back to \ref{coordinates-removed}.}}with $n \geq 3$ and a family of sets $\{ N_i \}_{i \in I} \subset \mathbb{R}^n$, where $I$ is a finite index set. Each of the sets $N_i$ \hypertarget{target:coordinates-removed}{}\correction{is homeomorphic to a Cartesian product of two balls of dimension $d_x$ and $d_y$, respectively.\mymarginpar{\hyperlink{issue:hset}{Link back to \ref{issue:hset}.}}\hypertarget{target:hset}{} We refer to these sets as {\em h-sets} (see Definition \ref{def:hset}).
We refer to $d_x$ as the {\em topological exit} dimension and to $d_y$ as the {\em topological entry} dimension.} We say that $N_i$ \emph{$f$-covers}  $N_j$, which we write $N_i \cover{f} N_j$, if the image set $f(N_i)$ \correction{is mapped in a topologically good way onto $N_j$ as formally expressed in Definition~\ref{def:covering}}. Figure~\ref{fig:covering} provides an illustration for the case $n = 3$. \correction{Informally, a covering relation involves a topological `stretching' of an h-set by the map $f$ along the topological exit direction, and a topological `squeezing' by the map $f$ along the topological entry direction, with relation to the second h-set that is being covered.}

The family of h-sets $\{ N_i \}_{i \in I}$ is equipped with a \emph{cone field} and the map $f$ satisfies a cone condition, meaning that the cone field is mapped into itself; see already Definition~\ref{def:cone-cond} for the precise statement of cone conditions.
We consider a family of topological (hyper)disks that are aligned with this cone field; we refer to them as \emph{horizontal disks} and they are defined formally in Definition~\ref{def:horizontal-disk}. With these notions, we are able to state our central result.

\hypertarget{target:mother}{}\mymarginpar{\hyperlink{issue:mother}{Link back to \ref{issue:mother}}}
In our setting, we also consider second family of h-sets $\{ M_l \}_{l \in L}$, which we refer to as `mother sets'. \correction{The mother sets play a special role in our construction: a mother set does not need to cover any other h-set. We make the distinction in notation in order to tell apart the `large' mother sets from the `smaller' sets that cover other sets.}

\begin{figure}[t!]
  \centering
  \includegraphics{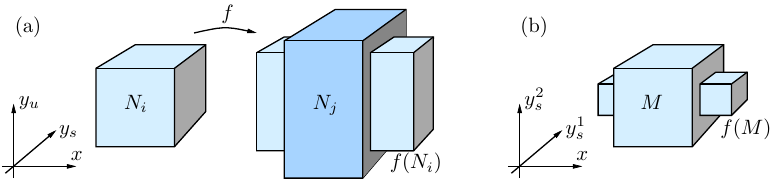}
  \caption{\label{fig:covering}
    Illustration of covering relations in a three-dimensional space. Panel~(a) shows a covering $N_i \cover{f} N_j$ with \correction{topological exit direction} $x$ and \correction{topological entry direction} $y =(y_u,\, y_s)$, assuming (weak) expansion in the direction $y_u$ and contraction along $y_s$. Panel~(b) shows the case $M \cover{f} M$ of a box that $f$-covers itself, as is typically found near a hyperbolic fixed point with one unstable direction $x$ and two stable directions $y^1_s$ and $y^2_s$. \correction{In both~(a) and~(b) the topological exit dimension is $d_x=1$ and the topological entry dimension is $d_y=2$.}}
\end{figure}

\begin{theorem}[Characterisation of a blender in terms of covering relations]
\label{th:main}
Let $f : \mathbb{R}^n \to \mathbb{R}^n$ be a diffeomorphism, with $n \geq 3$, and assume that there exists a set $U\subset \mathbb{R}^n$ such that its \correction{maximal} invariant set $\Lambda$ is hyperbolic and transitive.
\correction{Consider two finite families of h-sets $\{ N_i \}_{i \in I}\subset U$ and $\{ M_l \}_{l \in L} \subset U$} that satisfy the following conditions:
\begin{list}{(B\arabic{enumi})}{\usecounter{enumi} \setlength{\leftmargin}{9mm} \setlength{\labelwidth}{9mm}}
\item For every $l \in L$ and every horizontal disk $[h]$ in $M_l$, there exist $i \in I$ such that $[h] \cap N_i$  is a horizontal disk in $N_i$.
\item \correction{For every $i \in I$ at least one of the following two cases holds:
\begin{itemize}[leftmargin=-1em]
\item there exists $j \in I$ such that $N_i \cover{f} N_j$ and $f$ satisfies cone conditions from $N_i$ to $N_j$,
\item there exists $l \in L$ such that $N_i \cover{f} M_l$ and $f$ satisfies cone conditions from $N_i$ to $M_l$.
\end{itemize}
}
\end{list}
If, in addition, the \correction{dimension of the topological entry} of the h-sets is larger than the \correction{dimension of the stable bundle} of $\Lambda$, then $\Lambda$ is a blender.
\end{theorem}

\correction{For readers who prefer more abstract statements we provide a shorter reformulation of the above result in Theorem \ref{cor:main}, where we state it without referring to mother sets. However, we find the notational distinction of the mother sets helpful for the actual construction leading to {\em(B1)} and {\em(B2)}; we discuss this in Section \ref{sec:compsetup}.}

The two conditions \emph{(B1)} and~\emph{(B2)} provide us with a flexible framework for establishing the existence of blenders that can be applied in various contexts. They concern a neighbourhood of a hyperbolic set, and specify when such a hyperbolic set is a blender. Note that the simplest case of satisfying \emph{(B1)} and~\emph{(B2)} is a set $M$ that $f$-covers itself, which is a typical setting one encounters 
around a hyperbolic fixed point; see Figure~\ref{fig:covering}(b). \hypertarget{target:cardinality}{}\correction{While \eqref{eq:key-objective} is satisfied in this case, there is no blender in $M$ since the dimension of topological entry is the same as the dimension of the stable bundle.}\mymarginpar{\hyperlink{issue:cardinality}{Link back to \ref{issue:cardinality}}}

In general, to establish the existence of a blender, one should expect to need a possibly large number of h-sets $\{ N_i \}_{i \in I}$ and $\{ M_l \}_{l \in L}$. Nevertheless, verifying \emph{(B1)} and~\emph{(B2)} only requires checking bounds on first images of the sets $N_i$ to establish the required covering and cone conditions.
\correction{Since covering relations and cone conditions are persistent under $C^1$-perturbations, blenders established with our method are indeed $C^1$-persistent.}

\correction{Conditions in the same spirit as {\em(B1)} and {\em(B2)} can also be used to establish the existence of $C^1$-persistent heterodimensional cycles}~\cite{bonatti2008, bonatti2012Nonlin}.
\hypertarget{target:cycles}{}\mymarginpar{\hyperlink{issue:cycles}{Link back to \ref{issue:cycles}}}

\correction{\begin{definition}[Heterodimensional cycle]\cite{bonatti2008}
Let $\Lambda $ and $\widetilde{\Lambda}$ be two hyperbolic sets of a diffeomorphism $f : \mathbb{R}^n \to \mathbb{R}^n$.
We say that $f$ has a heterodimensional cycle if
\begin{enumerate}
\item  the indices (dimensions of the unstable bundle) of the sets $\Lambda$ and $\widetilde{\Lambda}$ are different,
\item the stable manifold of $\Lambda$ meets the unstable manifold of $\widetilde{\Lambda}$ and the same holds for the stable
manifold of $\widetilde{\Lambda}$ and the unstable manifold of $\Lambda$.
\end{enumerate}
\end{definition}}

\correction{The conditions required to give $C^1$-persistence of a heterodimensional cycle will follow from the assumption that we have suitable homoclinic and heteroclinic connections between the two hyperbolic invariant sets. These connections are expressed in terms of sequences of covering relations between two families of h-sets. One family of h-sets is positioned at $\Lambda$ and the other at $\widetilde{ \Lambda}$. We will assume that these two families are {\em interconnected}; the formal definition of this notion is spelled out in Definition \ref{def:interconnected}. Then we have the following result.}

\correction{\begin{theorem}[Existence of a heterodimensional cycle]
\label{th:cycles}
Consider two sets $U,\widetilde{U}\subset \mathbb{R}^n$ with two hyperbolic invariant sets for a diffeomorphism $f : \mathbb{R}^n \to \mathbb{R}^n$,
\[
\Lambda = \mathrm{Inv}(f,U)\qquad\mbox{and}\qquad \widetilde{\Lambda} = \mathrm{Inv}(f,\widetilde{U}).
\]
Let $\{ M_l \} \subset U$ and $\{ \widetilde{M}_i \} \subset \widetilde{U}$ be two finite families of h-sets. Consider the following condition:\begin{list}{(B\arabic{enumi})}{\usecounter{enumi}
    \setlength{\leftmargin}{9mm} \setlength{\labelwidth}{9mm}}
\setcounter{enumi}{2}
\item The families $\{ M_l \} $ and $\{ \widetilde{M}_i \}$ are {\em interconnected} by sequences of covering relations.
\end{list}
If (B3) is satisfied and the indices of the sets $\Lambda$ and $\widetilde{\Lambda}$ are different, then $f$ has a heterodimensional cycle.
\end{theorem}}
\correction{Condition {\em (B3)} is expressed in terms of covering relations and cone conditions, which are persistent under $C^1$-perturbations (see Definition \ref{def:interconnected}). This means that the heterodimensional cycles from Theorem \ref{th:cycles} are persistent under small $C^1$-perturbations of the map.}

\hypertarget{target:homoclinic}{}\mymarginpar{\hyperlink{issue:homoclinic}{Link back to \ref{issue:homoclinic}}}
\mymarginpar{\hyperlink{issue:2blenders}{Link back to \ref{issue:2blenders}}}
\correction{Condition {\em(B3)} requires homoclinic connections to $\Lambda$ and $\widetilde{\Lambda}$, as well as heteroclinic connections between $\Lambda$ and $\widetilde{\Lambda}$, expressed by means of suitable sequences of covering relations. If the first of the two hyperbolic sets, say, $\Lambda$, is a blender as validated by means of Theorem \ref{th:main}, then one of the homoclinic connections needed for {\em (B3)} follows from {\em (B1)} and {\em (B2)}. In other words, it is merely a reformulation of blender activation. In fact, to fulfil {\em (B3)} and to have different indices of the hyperbolic sets at least one of them needs to be a blender. The heteroclinic connections in {\em (B3)} ensure that we have a link from the blender $\Lambda$ to the second hyperbolic set $\widetilde{\Lambda}$, and back. This means that the unstable manifold of the $\widetilde{\Lambda}$ can enter a neighbourhood of $\Lambda$, which ensures an intersection with the stable manifold of the blender $\Lambda$. The heteroclinic connection in the other direction ensures that the unstable manifold of the blender $\Lambda$ intersects the stable manifold of $\widetilde{\Lambda}$. Indeed, $\widetilde{\Lambda}$ may be a hyperbolic fixed point, but it could also be a hyperbolic periodic orbit, or some more complicated hyperbolic set, or even a blender.} 

\correction{Our aim is also to show that our approach and tools can be used in practice. There are only a few known examples of blenders in diffeomorphisms given by explicit formulas.} Perhaps the most studied example is the H{\'e}non-like family 
\begin{equation}
\label{eq:Henon-family}
  f\left(\mathrm{x},\, \mathrm{y},\, \mathrm{z}\right) =
  \left( \mathrm{y},\, \mu + \mathrm{y}^2 + \beta \, \mathrm{x}, \xi \, \mathrm{z}+\mathrm{y} \right)
\end{equation}
that was studied in~\cite{hkos_blenderDEA, hkos_blender, hkos_boxedin}; see also~\cite{diaz2014}. An analytical proof of existence was given in~\cite{hkos_blender} for the single parameter point $\mu = -9.5$, $\beta = 0.1$ and $\xi = 1.185$, along with a numerically generated conjecture that a blender exists over the full range of $\xi \in (1, 1.843]$.  Similarly, the numerical study in~\cite{hkos_blenderDEA} suggests that a blender exists for $\mu = -9.5$, $\beta = 0.3$ and $\xi \in (1, 1.75]$, which we select as an illustration of how to apply Theorem~\ref{th:main}. More precisely, in Section~\ref{sec:CAP}, we present a proof of the following theorem.


\begin{theorem}
\label{th:henon}
Consider the H{\'e}non-like family defined in~\eqref{eq:Henon-family} with $\mu = -9.5$ and $\beta = 0.3$. For every $\xi \in [ 1.01,\, 1.125]$, this H{\'e}non-like family has a blender.
\end{theorem}


There are some differences between our approach and the classical setup by Bonatti and D\'{\i}az~\cite{MR1381990}. Their method is based on appropriate assumptions on the images of topological disks under iterates of the map, in a neighbourhood of a horseshoe at a hyperbolic fixed point. Our method can be applied in this setting, but is not limited to it. More precisely, fixed points or periodic points and their associated invariant manifolds are not used explicitly in our approach. This means that our method can be applied more generally, for example, to hyperbolic sets that arise in the context of hetero/homoclinic tangles between a number of hyperbolic fixed points. In particular, our construction of covering provides a method for verifying the existence of \correction{$C^1$-persistent} heterodimensional cycles.

Another difference is that we do not require conditions on the iterations of disks. Rather, what we require is good topological alignment of images of sets, expressed through covering relations. Covering relations have proven to be a versatile and flexible tool, due to the simplicity of their validation~\cite{MR3032848, CG-CPAM, MR2947932, MR1862804, MR2351028, MR2174417,  MR2534406}. Moreover, we can position our sets in such a way that we have a sequence of coverings, and validation requires only a single iterate of the map instead of considering images of topological disks for several compositions of the map, as in~\cite{MR1381990}. Therefore, we can use a `topological parallel shooting' approach, that allows us to avoid compositions. In particular, our assumptions are formulated by means of conditions that can be validated by using interval arithmetic and computer-assisted tools. The proof in~\cite{hkos_blender} is for existence of a blender that is a \emph{maximal} (transitive) hyperbolic set. Our approach is, in some sense, more general and, as an example, we present in Section~\ref{sec:CAP} the proof of existence of a blender in the H{\'e}non-like family for the parameter range as stated in Theorem~\ref{th:henon}; this blender is contained in a family of covering sequences that are positioned to be disjoint from any fixed point.

In our construction we use a \emph{wall property}, which states that a set in the investigated neighbourhood intersects with every horizontal disk. During our construction we ensure that the wall property is satisfied by the stable set of the invariant hyperbolic set, and this implies that the hyperbolic set is a blender.

The paper is organised as follows. Section~\ref{sec:prel} introduces some necessary background. Specifically we present the affine construction of a blender in Section~\ref{sec:blender-aff}. We then provide required preliminaries in Section~\ref{sec:cone-cond-covering} by introducing notation and statements regarding cone conditions and covering relations. In Section~\ref{sec:main-results}, we prove our main Theorem \ref{th:main}, and we also explain conditions \emph{(B1)} and~\emph{(B2)} more precisely. In Section~\ref{sec:heteroclinic}, \correction{we explain the details of condition~\emph{(B3)} and prove Theorem \ref{th:cycles}}. In Section~\ref{sec:CAP}, we present a computer-assisted proof\footnote{The code for the computer assisted proof is available on the personal web page of the corresponding author.} of Theorem~\ref{th:henon} regarding the existence of a $2$-blender for the H\'enon-like family \eqref{eq:Henon-family}. This example illustrates how our conditions can be validated in practice. We end with a discussion and conclusions in Section~\ref{sec:conclusions}. Furthermore, in the Appendix, we also show how hyperbolicity and transitivity of an invariant set can be established with computer-assisted tools.

\section{Background, notation and preliminaries}
\label{sec:prel}
\correction{We start with basic preliminaries for hyperbolic sets in Section~\ref{sec:hyperbolic-set}. Then, we discuss a simple affine example of a blender in Section~\ref{sec:blender-aff}. We do so in order to provide a simple illustration and intuition for our method.} The general discussion of our two main tools, cone conditions and covering relations follows in Section~\ref{sec:cone-cond-covering}.

\subsection{Hyperbolic sets}
\label{sec:hyperbolic-set}

\correction{Here, we recall the standard definitions of a hyperbolic set and its invariant manifolds.}


\begin{definition}
\label{def:hyp}
The set $\Lambda \subset \mathbb{R}^n$ is hyperbolic if for every $p \in \Lambda$, the $n$-dimensional tangent space $T_p\Lambda$ splits into a direct sum $T_p\Lambda = E_p^u \oplus E_p^s$ of (possibly trivial) vector spaces (fibre bundles) $E_p^u$ and $E_p^s$, respectively, that are invariant under the derivative $Df$ of $f$. More precisely, for nonzero vectors $v_u \in E_p^u$ and $v_s \in E_{p}^s$, we have
\begin{align*}
  Df(p) \, v_u &\in E_{f(p)}^u,  \\
  Df(p) \, v_s &\in E_{f(p)}^s,
\end{align*}
%
and there exist $c > 0$ and $\lambda \in (0,1)$, independent of $p$, such that, for all $j \geq 1$,
\begin{align*}
  \left\| Df^{-j}(p) \, v_u \right\| &< c \, \lambda^j \, \left\| v_u \right\|, \\
  \left\| Df^j(p) \, v_s \right\| &< c \, \lambda^j \, \left\| v_s \right\|.
\end{align*}
\end{definition}


\correction{\begin{definition}[Unstable and stable manifolds of $\Lambda$]
\label{def:inv-manifolds}
For a given neighbourhood $U$ of a hyperbolic set $\Lambda$ the local unstable and stable manifolds of $\Lambda$ in $U$ are defined as
\begin{displaymath}
  W^u_{\mathrm{loc}}(\Lambda) = \bigcup_{p\in\Lambda} W_{p,\mathrm{loc}}^u \quad \text{and} \quad
  W^s_{\mathrm{loc}}(\Lambda) = \bigcup_{p\in\Lambda} W_{p,\mathrm{loc}}^s,
\end{displaymath}
where $W_{p,\mathrm{loc}}^u$ and $W_{p,\mathrm{loc}}^s$ are the local unstable and stable fibres, respectively, defined as
\begin{align*}
  W_{p,\mathrm{loc}}^u  &= \left\{ q \in \mathbb{R}^n \bigm|  \| f^{-j}(p) - f^{-j}(q)\| \to 0 \ \text{as} \ j \to \infty,\mbox{ and }  f^{-j}(q) \in U \mbox{ for all } j\ge 0\right\}, \\
  W_{p,\mathrm{loc}}^s  &= \left\{ q \in \mathbb{R}^n \bigm|  \| f^j(p)  - f^j(q) \| \to 0 \ \text{as} \  j \to \infty,\mbox{ and }  f^{-j}(q)\in U \mbox{ for all } j\ge 0 \right\}.
\end{align*}
We also define the respective global unstable and stable manifolds as
\begin{displaymath}
  W^u(\Lambda) = \bigcup_{k\ge 0} f^k(W^u(\Lambda)) \quad \text{and} \quad
  W^s(\Lambda) = \bigcup_{k\le 0} f^k(W^s(\Lambda)).
\end{displaymath}
\end{definition}}

\begin{figure}[t!]
  \centering
  \includegraphics{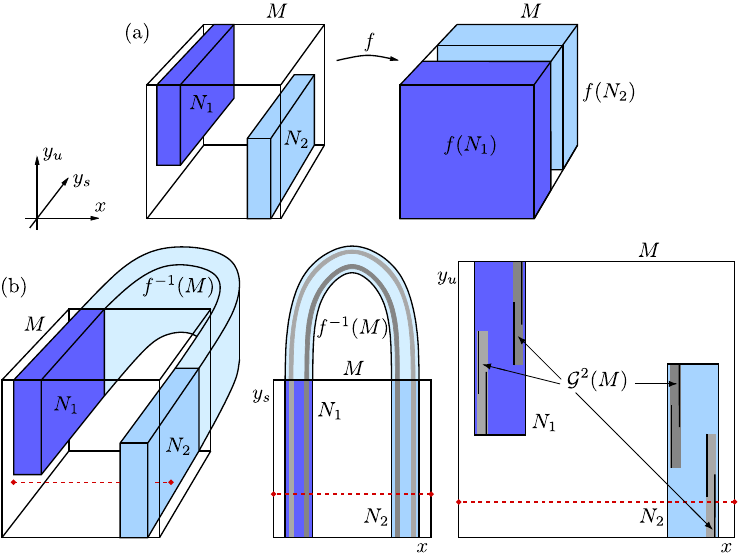}
  \caption{\label{fig:blender_affine}
    Sketch of how a blender horseshoe is generated when the map $f$ is affine in a box $M$ defined on the three-dimensional $(x,y_s,y_u)$-space. Panel~(a) illustrates how sub-boxes $N_1$ (blue) and $N_2$ (light blue) and their images are contained in the mother set $M$. Panel~(b) shows the three-dimensional horseshoe set $f^{-1}(M)$ that generates $N_1 \dot\cup N_2 \correction{= \mathcal{G}(M):=f^{-1}(M)\cap M}$, with a top view in the $(x,y_s)$-plane that also shows \correction{$f^{-1}(\mathcal{G}(M))$} (grey set) and a front view in the $(x,y_u)$-plane that \correction{shows $\mathcal{G}^2(M)$ in grey and $\mathcal{G}^3(M)$ in black}. Note that any horizontal segment, such as the dashed red line, that traverses $M$ in the $x$-direction must intersect \correction{$\mathcal{G}^{j}(M)$} for any $j \geq 1$.}
\end{figure}

\subsection{The affine blender horseshoe}
\label{sec:blender-aff}
Bonatti, Crovisier, D\'{\i}az, and Wilkinson designed an abstract example in~\cite{MR3559345} of a diffeomorphism $f$ on $\mathbb{R}^3$ that features a blender horseshoe. We briefly discuss this example to demonstrate the features of our construction of covering relations, as well as provide intuition behind it. We focus on the geometric features, which are illustrated in Figure~\ref{fig:blender_affine}. The coordinate axes are denoted $x$, $y_u$, and $y_s$, along which $f$ is expanding, weakly expanding, and contracting, respectively. The diffeomorphism $f$ maps two cuboids $N_1$ and $N_2$, which are contained in a larger cube $M$, tightly back into $M$ as shown in Figure~\ref{fig:blender_affine}(a). Here, the restrictions $f|_{N_1}$ and $f|_{N_2}$ are taken to be affine and the coordinate axes agree with the expanding and contracting directions of $f$.

A better intuition can be gained by considering the pre-images of these sets, as illustrated in Figure~\ref{fig:blender_affine}(b). This reveals the three-dimensional blender horseshoe and highlights the particular selection of $N_1$ and $N_2$ such that $N_1 \cup N_2 =\correction{\mathcal{G}(M):=} f^{-1}(M) \cap M$. Indeed, the top view of the $(x,y_s)$-plane shows the well-known affine Smale horseshoe construction inside a square~\cite{smale}. The (forward and backward) invariant set of $f$ in $M$ is a hyperbolic Cantor set, which we denote $\Lambda$. Further pre-images \correction{$\mathcal{G}^{j}(M)$} form a nested sequence of cuboids in $M$, with a limit equal \correction{to the local stable manifold $W^s_{\mathrm{loc}}(\Lambda)$ in $M$. Due to the affine nature of $f$, the local stable manifold $W^s_{\mathrm{loc}}(\Lambda)$ is a Cantor set of line segments that stretch all the way across in the $y_s$-direction.}

Notice from the projections onto the $(x,y_u)$-plane in Figure~\ref{fig:blender_affine}(b) that the cuboids $N_1$ and $N_2$ overlap with respect to the weakly expanding $y_u$-direction. As a result, any `horizontal' segment in the $x$-direction that joins the left and right faces of $M$ must intersect at least one of the preimages in \correction{$\mathcal{G}^{j}(M)$}, for any given $j$; see the (red) dashed line as an example of such a horizontal segment in panel~(b). It follows from taking the limit as $j \to \infty$ that the given horizontal segment must intersect \correction{$W^s_{\mathrm{loc}}(\Lambda)$ in $M$}. This is the surprising and characterising feature of the blender, which can be interpreted as \correction{$W^s_{\mathrm{loc}}(\Lambda)$} behaving effectively like an object of a higher dimension, a surface in this case, when viewed along the $x$-direction. Importantly, the existence of such an intersection \correction{persists under small $C^1$-perturbation of the map}.
\hypertarget{target:pers}{}
\mymarginpar{
\hyperlink{issue:pers}{Link back to \ref{issue:pers}}
} 

\subsection{Cone conditions, covering relations and horizontal disks}
\label{sec:cone-cond-covering}
Our objective is to develop a methodology for establishing the existence of blenders within an explicit domain and providing an explicit estimate for the set of surfaces $\mathcal{S}$ that intersect $W^s(\Lambda)$. Our main two tools for doing so are cone conditions and covering relations.

More specifically, we wish to establish the existence of a blender that intersects a finite union
\begin{displaymath}
  \bigcup N_i := \bigcup\limits_{i \in I} N_i \subset \mathbb{R}^n
\end{displaymath}
of compact sets $N_i\subset\mathbb{R}^n$ with index set $I$. We equip each $N_i$ with a local coordinate change $\gamma_i : \mathbb{R}^n \to \mathbb{R}^n$ that maps points in $N_i$ to pairs $(x,\, y) \in \mathbb{R}^{d_x} \times \mathbb{R}^{d_y} = \mathbb{R}^n$, with the components $x \in \mathbb{R}^{d_x}$ and $y \in \mathbb{R}^{d_y}$ representing `topological exit' and `topological entry' coordinates, and corresponding projections $\pi_{x}$ and $\pi_{y}$, respectively. We refer to $N_i$ as an \textit{h-set}  (which stands for hyperbolic-type set) that is formally defined as follows; see also Figure~\ref{fig:covering}.

\begin{definition}[h-set \cite{MR2060532,MR2060531}]
\label{def:hset}
A compact subset $N \subset \mathbb{R}^n$ is an \emph{h-set} if there exists a homeomorphism $\gamma: \mathbb{R}^n \to \mathbb{R}^n$ such that
\begin{equation}
\label{eq:N-local}
  N_{\gamma} := \gamma(N) = \overline{B}_{d_x} \times \overline{B}_{d_y},
\end{equation}
where $\overline{B}_{d_x}$ and $\overline{B}_{d_y}$ are closed unit balls; \correction{we refer to their dimensions $d_x$ and $d_y$ as the dimensions of the topological exit and entry, respectively.} We define the \emph{exit set} $N^-$ and the \emph{entry set} $N^+$ as
\begin{displaymath}
  N^{-} :=\gamma^{-1}\left( \partial \overline{B}_{d_x} \times \overline{B}_{d_y} \right)
  \quad \text{and} \quad
  N^{+} :=\gamma^{-1}\left( \overline{B}_{d_x} \times \partial\overline{B}_{d_y} \right).
\end{displaymath}
\end{definition}

\begin{remark}
\label{rem:norms}
To define the closed balls $\overline{B}_{d_x}$ and $\overline{B}_{d_y}$ in \eqref{eq:N-local} we can use different norms on $\mathbb{R}^{d_x}$ and $\mathbb{R}^{d_y}$, respectively. For applications in computer-assisted proofs, a natural choice is to use the maximum norm, in which case the balls are Cartesian products of intervals, or simply (hyper)cubes.
\end{remark}

From now on, we assume that the $N_i$ are h-sets for all $i \in I$. The dimensions $d_x$ and $d_y$ do not depend on the index $i \in I$. For simplicity we refer to the local coordinates as $(x,\, y)$ regardless of the choice of $N_i$, keeping in mind that they are associated with a local coordinate change $\gamma_i$.

\begin{remark}
We do not need to assume that the dimensions $d_x$ and $d_y$ agree with the dimensions $d_u$ and $d_s$ of the unstable and stable fibres of $\Lambda$ from Definition~\ref{def:hyp}. In our methodology the exit coordinates $x$ are associated with a strong hyperbolic expansion, and the entry coordinates $y$ could include weak hyperbolic expansion.
\end{remark}

When the diffeomorphism $f$ is expressed in local coordinates of sets $N_i$ and $N_j$ for $i, j \in I$ then we refer to it as
\begin{displaymath}
  f_{ji} = \gamma_j \circ f \circ \gamma_i^{-1}.
\end{displaymath}
This representation is useful for the definition of a covering relation between two h-sets.

\begin{definition}[Covering between h-sets~\cite{MR2060532, MR2060531}]
\label{def:covering}
We say that the h-set \emph{$N_i$ $f$-covers the h-set $N_j$}, which we denote
\begin{displaymath}
  N_i \cover{f} N_j,
\end{displaymath}
if there exist:
\begin{enumerate}
\item
a homotopy $\varsigma : [0, 1] \times N_{\gamma_i} \to \mathbb{R}^{d_x} \times \mathbb{R}^{d_y}$, where $N_{\gamma_i} := \gamma_i(N_i)$, such that for any $(x, y) \in N_{\gamma_i}$ and any $t \in [0,1]$,
\begin{align*}
  \varsigma\left( \, 0,\, (x, y) \right) 
  &= f_{ji}(x, y), \\
  \varsigma\left( t,\, \gamma_i(N_i^{-}) \right) \cap N_{\gamma_j} &= \emptyset, \\
  \varsigma\left( t,\, N_{\gamma_i} \right) \cap \gamma_j(N_j^{+}) &= \emptyset,
\end{align*}
and
\item
a linear map $A : \mathbb{R}^{d_x} \to \mathbb{R}^{d_x}$ such that
\begin{align*}
  \varsigma\left( \, 1,\,  (x, y) \right) &= (A x,\, 0), \\
  A\left( \partial\overline{B}_{d_x} \right) & \subset \mathbb{R}^{d_x}\setminus{\overline{B}_{d_x}}.
\end{align*}
\end{enumerate}
\end{definition}

For our purpose of proving the existence of a blender, a covering $N_i \cover{f} N_j$ corresponds to a strong expansion for the $x$-component, while the $y$-component includes contraction as well as weak expansion. Hence, the picture in this case is as shown in Figure~\ref{fig:covering}(a). In particular, the situation in Figure~\ref{fig:covering}(b) of a set covering itself is excluded, because the  invariant set contained in $\bigcup N_i$ could then be a mere saddle fixed point. Such a saddle fixed point would have a ($d_s=2$)-dimensional stable manifold, so the only possible choice of $k > d_s$ is $k = 3$, and it is not possible to have a $3$-blender in $\mathbb{R}^3$.

\begin{figure}[t!]
  \centering
  \includegraphics{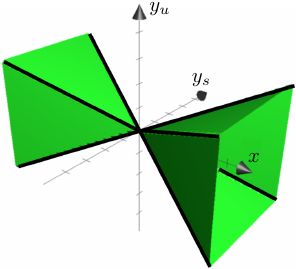}
  \caption{\label{fig:cone}
    Example of a cone ${\cal C}(0)$ defined for local coordinates of a three-dimensional space $\mathbb{R}^{d_x} \times \mathbb{R}^{d_y} = \mathbb{R}^1 \times \mathbb{R}^2$ with chosen norms $\| x \|_{d_x} = \frac{1}{2} \, | x |$ and $\| y \|_{d_y} = \| (y_u,\, y_s) \|_{d_y} = \max{( | y_u |,\, | y_s |)}$.}
\end{figure}
%
We quantify the level of expansion and contraction induced by a covering relation in terms of a cone condition defined for the respective h-sets. In the local coordinates given by $\gamma_i$, we define a cone at $z_{\gamma_i} = (\pi_x(z_{\gamma_i}), \pi_y(z_{\gamma_i})) \in \mathbb{R}^{d_x} \times \mathbb{R}^{d_y}$ as
\begin{equation}
\label{eq:cone-def-2}
  {\cal C}_{\gamma_i}(z_{\gamma_i}) := \left\{ (q_x,\, q_y) \in \mathbb{R}^{d_x} \times \mathbb{R}^{d_y} \Bigm| \left\| q_x - \pi_x(z_{\gamma_i}) \right\|_{d_x}  > \left\| q_y - \pi_y(z_{\gamma_i}) \right\|_{d_y} \right\},
\end{equation}
where $\| \cdot \|_{d_x}$ and $\| \cdot \|_{d_y}$ are norms on $\mathbb{R}^{d_x}$ and $\mathbb{R}^{d_y}$, respectively. Figure~\ref{fig:cone} shows an example of a cone in local coordinates defined on $\mathbb{R}^{3}$ with $d_x = 1$ and $d_y = 2$, where $\| \cdot \|_{d_x}$ and $\| \cdot \|_{d_y}$ are given by the respective (suitably scaled) $L_1$-norm.

\begin{remark}
The choice of  the norms $\|\cdot\|_{d_x}$ and $\|\cdot\|_{d_y}$ can be independent of the choice of the norms that determine the closed balls $\overline{B}_{d_x}$ and $\overline{B}_{d_y}$ in Definition~\ref{def:hset}.
\end{remark}

In the state space of the map $f$ the cone at $z \in N_i$ is defined as
\begin{equation}
\label{eq:Ci-def}
  {\cal C}_i(z) := \gamma_i^{-1}\left(  {\cal C}_{\gamma_i}(\gamma_i(z)) \right).
\end{equation}

\begin{remark}
We do not rule out the situation that $z \in N_i \cap N_j$ for $i \neq j$, in which case typically ${\cal C}_i(z) \neq {\cal C}_j(z)$. However, this will not present a problem in our methodology.
\end{remark}

\begin{definition}[Cone conditions]
\label{def:cone-cond}
Let $f$ be a diffeomorphism and let $N_i,N_j$ be two h-sets. We say that $f$ satisfies \emph{cone conditions from $N_i$ to $N_j$} if, for any $z \in N_i$ such that $f(z) \in N_j$, we have
\begin{equation}
\label{eq:cc2}
  f\left( {\cal C}_i(z) \right) \subset {\cal C}_j\left( f(z) \right).
\end{equation}
\end{definition}

Importantly, one can have cone conditions in the setting when there is no contraction along the entry coordinates $y$. What is needed is sufficiently strong expansion of the exit coordinates $x$, as is demonstrated with the following toy example.

\begin{example}
\label{ex:linear}
Consider the linear map $f(x,\, y_u, \,y_s)  = \left( 4 x,\, 2 y_u,\, \frac{1}{2} y_s \right)$ on $\mathbb{R}^{3}$, and let $\| \cdot \|$ be the standard Euclidean norm. We have $d_x = 1$ and $d_y =2$ with $y = (y_u,\, y_s)$, where $y_s$ is a stable and $y_u$ an unstable coordinate. Despite the expansion along the $y_u$-coordinate, cone conditions hold for any choice of $N_i, N_j \subseteq \mathbb{R}^3$. Indeed, let $z \in \mathbb{R}^3$ and let $q \in \mathbb{R}^3$ be such that $\left\| \pi_{x} (q - z) \right\| - \left\| \pi_{y} (q - z) \right\| = \left\| x \right\| - \left\| (y_u,\, y_s)  \right\| > 0$, where we use the notation $q - z = (x,\, y_u,\, y_s)$. Since $f$ is linear, we have
\begin{eqnarray*}
  \left\| \pi_{x} (f(q) - f(z)) \right\| - \left\| \pi_{y} (f(q) - f(z)) \right\|
  &=& \left\| \pi_{x} f(q - z) \right\| - \left\| \pi_{y} f(q - z) \right\| \\
  &=& \left\| 4 x \right\| -\left\| \left(  2 y_u,\,  \tfrac{1}{2} y_s\right)  \right\| \\
  &>& 2 \, \left(  \left\| x \right\| - \left\| \left( y_u, \,  \tfrac{1}{4} y_s \right)  \right\| \right) \\
  &\geq&  2 \, \left(  \left\| x \right\| - \left\| (y_u,\,  y_s) \right\| \right)  \\
  &>& 0.
\end{eqnarray*}
\end{example}

The cone condition is used to define topological disks of dimension $d_x$ that in local coordinates are `well aligned' with the strongly expanding $x$-component of $\gamma_i(N_i)$; we refer to such $d_x$-dimensional disks as horizontal disks.

\begin{definition}[Horizontal disk]
\label{def:horizontal-disk}
The $d_x$-dimensional surface $[h] \subset \mathbb{R}^n$ is a \emph{horizontal disk} in $N_i$ if in local coordinates it is a graph over the $x$-component and it satisfies a cone condition. More formally, the following two requirements hold:
\begin{enumerate}
\item
there exists a continuous function $h : \overline{B}_{d_x} \to \overline{B}_{d_y}$ such that
\begin{displaymath}
  \gamma_i([h]) = \left(  x,\, h(x) \right),
\end{displaymath}
and
\item
for every $x_1, x_2 \in \overline{B}_{d_x}$ with $x_1 \neq x_2$ we have
\begin{displaymath}
  \left( x_1,\, h(x_1) \right) \in {\cal C}_{\gamma_i}(x_2,\,  h(x_2)).
\end{displaymath}
\end{enumerate}
We say that $[h]$ is a horizontal disk in a family of h-sets $\{ N_i \}_{i \in I}$ if there exists $k \in I$ such that $[h]$ is a horizontal disk in $N_k$.
\end{definition}

The key tool needed to prove Theorem~\ref{th:main} is a result from~\cite{MR2494688}, which we state here using our notation.

\begin{theorem}[{\cite[Thm 7]{MR2494688}}]
\label{th:disk-propagation}
Suppose $f : \mathbb{R}^n \to \mathbb{R}^n$, with $n \geq 3$, is a diffeomorphism and we have a finite sequence
\begin{displaymath}
  N_{i_0} \cover{f} N_{i_1} \cover{f} \ldots \cover{f} N_{i_k}
\end{displaymath}
of covering relations $\{ N_{i_m} \}_{i_m \in I} \subset \mathbb{R}^n$, such that $f$ satisfies cone conditions from $N_{i_{m-1}}$ to $N_{i_m}$ for $m = 1, \ldots, k$. Then, if $[h_{i_0}]$ is a horizontal disk in $N_{i_0}$, there exists a horizontal disk $[h_{i_k}]$ in $N_{i_k}$, such that
\begin{displaymath}
  [h_{i_k}] = \left\{ f^{k}(z) \bigm|  z \in [h_{i_0}] \ \text{and} \ f^m(z) \in N_{i_m} \ \text{for} \ m = 1, \ldots, k \right\}.
\end{displaymath}
\end{theorem}

\section{Proof of Theorem~\ref{th:main}}
\label{sec:main-results}
To recap the setting, we consider a diffeomorphism $f$ on $\mathbb{R}^n$, a set $U\subset \mathbb{R}^n$ whose invariant set \[
\Lambda=\mathrm{Inv}(f,U):=\left\{z: f^k(z)\in U \mbox{ for all } k\in\mathbb{Z} \right\}
\] is hyperbolic and transitive, and a finite family of h-sets $\{ N_i \}_{i \in I}\subset U$, each of which with topological exit coordinates $x\in \mathbb{R}^{d_x}$ and topological entry coordinates $y = (y_u,\, y_s) \in \mathbb{R}^{d_y}$ as given by the respective local coordinate changes $\gamma_i$. Moreover, each h-set $N_i$ is equipped with cones, as given by~\eqref{eq:cone-def-2} in local coordinates and by~\eqref{eq:Ci-def} in the original coordinates. We also consider a second family of h-sets $\{ M_l \} _{l \in L}\subset U$ of `mother sets'. \correction{The mother sets will not cover any other h-sets, but will be covered by the h-sets from the family $\{ N_i \}$. We assume that the dimensions $d_x$ of the topological exit and $d_y$ of the topological entry are the same for all the h-sets}. The claim of Theorem~\ref{th:main} is that $\Lambda$ is a $d_y$-blender, provided conditions~\emph{(B1)} and~\emph{(B2)} are satisfied.

Before we present the proof of Theorem~\ref{th:main}, we provide some intuition behind conditions~\emph{(B1)} and~\emph{(B2)}. The local coordinates for the coverings and cone conditions in~\emph{(B1)} and~\emph{(B2)} are in terms of a splitting $x \in \mathbb{R}^{d_x}$ and $y \in \mathbb{R}^{d_y}$. The hyperbolic nature of the invariant set, on the other hand, is in terms of the associated unstable and stable fibres of dimensions $d_u$ and $d_s$, respectively. We will sometimes use the phrase that conditions~\emph{(B1)} and~\emph{(B2)} hold with $(d_x,\, d_y,\, d_u,\, d_s)$ to state these dimensions specifically. We require \correction{$d_y > d_s$}, which implies that the $y$-component in local coordinates can be split into $y = (y_u,\, y_s)$, such that the unstable/expanding coordinates are $(x,\, y_u) \in \mathbb{R}^{d_u}$ and the stable/contracting coordinates are $y_s \in \mathbb{R}^{d_s}$.
\begin{figure}[t!]
  \centering
  \includegraphics{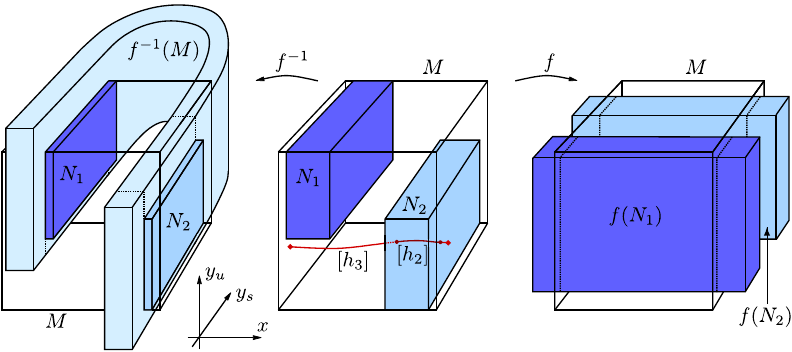}
  \caption{\label{fig:blender_modified}
    A modification of the affine blender from Figure~\ref{fig:blender_affine} so that $N_1 \cover{f} M$ and $N_2 \cover{f} M$; also shown is a horizontal disk $[h_3]$ (red curve) of $M$, which generates a horizontal disk $[h_2]$ (dark red curve) of $N_2$.}
\end{figure}

\begin{remark}
The horizontal disks $[h]$ from condition~(B1) play the role of the set ${\cal S}$ of \correction{disks} from Definition~\ref{def:blender}. More specifically, for the hyperbolic set $\Lambda$, we show that each horizontal disk $[h]$ in $\left\{N_i \right\}_{i\in I}$ intersects  \correction{$W^s_{\mathrm{loc}}(\Lambda)$}. Note that dimensions of the horizontal disks $[h]$ and the stable manifold of $\Lambda$ do not add up to the dimensions of the phase space, which is the characterising feature of a blender.
\end{remark}
%
In $\mathbb{R}^{3}$, the horizontal disks are (one-dimensional) curve segments and an example segment is shown in Figure~\ref{fig:blender_affine}(b); \correction{here $d_y>d_s$}. On the other hand, in Figure~\ref{fig:covering}(b), for \correction{$d_y= d_s$}, conditions \emph{(B1)} and~\emph{(B2)} are trivially satisfied by the general statement that $\Lambda  = W^u(\Lambda) \cap W^s(\Lambda)$. In fact $\Lambda$ could be a fixed point in this case, but to have a blender we must have \correction{$d_y> d_s$}, which means that this case is ruled out and not part of our considerations. Figure~\ref{fig:blender_modified} shows how the example of the affine blender from Section~\ref{sec:blender-aff} and Figure~\ref{fig:blender_affine} can be modified to fit into the framework of Theorem~\ref{th:main}. A suitable small reduction of the set $M$ achieves that the h-sets $N_1$ and $N_2$ cover $M$, as is shown in Figure~\ref{fig:blender_modified} with sketches of images and preimages. There, $N_1 \cover{f} M$ and $N_2 \cover{f} M$, which ensures~\emph{(B2)}. Condition~\emph{(B1)} is also satisfied, since every horizontal disk $[h_3]$ in $M $, which is a curve segment aligned with the $x$-direction here, must intersect $N_1$ or $N_2$. The intersection set is, hence, a horizontal disk in $N_1$ or $N_2$, provided the cones in $M$ are at least as sharp as the cones in the sets $N_i$. As is sketched in Figure~\ref{fig:B1}(a) for $[h_2]$ in $N_2$, this can be ensured by an appropriate choice of local coordinates; Figure~\ref{fig:B1}(b) shows an example for a more general h-set $N_i$.
%
\begin{figure}[t!]
  \centering
  \includegraphics{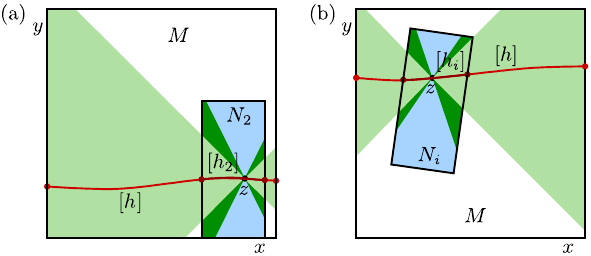}
  \caption{\label{fig:B1}
    To satisfy conditions \emph{(B1)} and~\emph{(B2)}, the cones in $M$ (light green), need to be contained in the cones in $N_i$ (dark green) at every point of $z \in [h_i]$; this is illustrated in panel~(a) for $N_2 \subset M$ of the affine example from Figure~\ref{fig:blender_modified}, and in panel~(a) for a general h-set $N_i \subset M$.}
\end{figure}

The following notion of a wall is used as a building block for the proof of Theorem~\ref{th:main}, and it is illustrated in Figure~\ref{fig:wall} for the affine case from Section~\ref{sec:blender-aff}.

\begin{figure}[t!]
  \centering
  \includegraphics{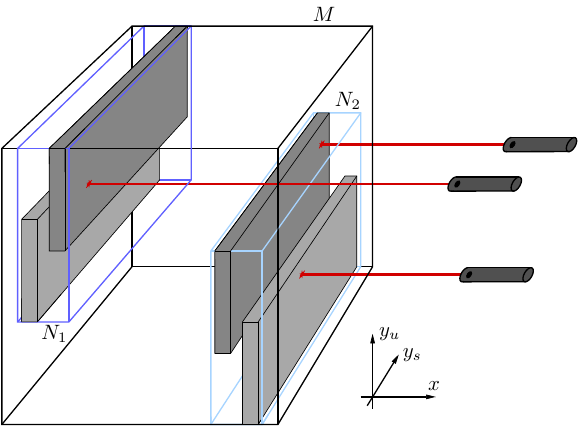}
  \caption{\label{fig:wall}
    Illustration of a wall in $\{ N_1,\, N_2,\, M \}$ formed by $f^{-2}(M) \cap M$ (light- and dark-grey sets). Any laser beam aligned with the $x$-direction, even if it were to take a slightly curvy path, cannot pass through the wall; compare with Figure~\ref{fig:blender_affine}(b).}
\end{figure}

\begin{definition}[Wall for a family of h-sets]
\label{def:wall}
\correction{A set $A \subset \mathbb{R}^n$ is a \emph{wall} for the two families of h-sets $\{ N_i \} _{i \in I} $ and $\{ M_l \} _{l \in L} $ if every horizontal disk $[h]$ in  $\{ N_i \}_{i \in I} $ intersects $A$ and also every horizontal disk $[h]$ in  $\{ M_l \}_{l \in L} $ intersects $A$.} 
\end{definition}

Figure~\ref{fig:wall} shows a mother set $M$ with two subsets, denoted $N_1$ and $N_2$; see also Figure~\ref{fig:blender_modified}. Any horizontal disk $[h] $ in $ \{ N_1,\, N_2,\, M \}$, which are curve segments aligned with the $x$-direction, intersects at least one of the light- and dark-grey boxes in $f^{-2}(M) \cap M$. Hence, $f^{-2}(M) \cap M$ is a wall for $\{ N_1,\, N_2,\, M \}$. The following two lemmas emphasise useful consequences of having a wall for a family of h-sets $\{ N_i \} _{i \in I}$  \correction{and $\{ M_l \} _{l \in L}$} and how to construct a wall if conditions \emph{(B1)} and~\emph{(B2)} from Theorem~\ref{th:main} are satisfied.

\begin{lemma}
\label{lem:wall-propagation}
If $A$ is a wall for the \correction{two families of h-sets $\{ N_i \} _{i \in I}$ and $\{ M_l \} _{l \in L}$,} and conditions (B1) and~(B2) from Theorem~\ref{th:main} are satisfied then
\begin{displaymath}
  \mathcal{G}\left( A \right) := f^{-1}\left( A \right) \cap \correction{\left(\bigcup N_i \cup \bigcup M_l \right)}
\end{displaymath}
is also a wall for $\{ N_i \} _{i \in I}$ \correction{and $\{ M_l \} _{l \in L}$}.
\end{lemma}

\begin{proof}
First, let us take a horizontal disk $[h_i]$ in $N_i$, for any $i\in I$. Our objective is to show that $[h_i] \cap \mathcal{G}\left( A \right) \neq \emptyset$. By~\emph{(B2)} we have $N_i \cover{f} N_j$ for some $j \in I$ and $f$ satisfies cone conditions from $N_i$ to $N_j$, \correction{or $N_i \cover{f} M_l$ and $f$ satisfies cone conditions from $N_i$ to $M_l$}. \correction{In the first case that $N_i \cover{f} N_j$}, by Theorem~\ref{th:disk-propagation} (applied for $k=1$) we obtain a horizontal disk $[h_j]$ in $N_j$ such that
\begin{displaymath}
  [h_j] =\left\{ f(z)  \bigm| z \in [h_i] \mbox{ and } f(z)\in N_j\right\} .
\end{displaymath}
Since $A$ is a wall, there exists $z \in [h_j] \cap A \neq \emptyset$, which implies that $f^{-1}(z) \in f^{-1}(A) \cap N_i \subset \mathcal{G}\left( A \right) $; hence, $[h_i] \cap \mathcal{G}\left( A \right) \neq \emptyset$, as required. \correction{In the second case that $N_i \cover{f} M_l$, we have a mirror argument:  by  Theorem~\ref{th:disk-propagation}
\begin{displaymath}
  [h_l] =\left\{ f(z)  \bigm| z \in [h_i] \mbox{ and }f(z)\in M_l\right\}
\end{displaymath}
is a horizontal disk in $M_l$; since $A$ is a wall, it intersects $M_l$, which implies that $[h_i] \cap \mathcal{G}\left( A \right) \neq \emptyset$.
}

Now consider a horizontal disk $[h]$ in $M_l$ for some $l \in L$. Our objective is to show that $[h] \cap \mathcal{G}\left( A \right) \neq \emptyset$. From condition~\emph{(B1)} we can choose $i\in I$ and a horizontal disk $[h_i]$ in $N_i$ such that $[h] \cap N_i = [h_i]$. 
As in the first part of the proof, we that conclude $[h_i] \cap \mathcal{G}\left( A\right) \neq \emptyset$; since $[h_i] \subset [h]$, this implies that $[h] \cap \mathcal{G}\left( A \right) \neq \emptyset$, as required.
\end{proof}

\begin{lemma}
\label{lem:Wall}
Let \correction{$\{ N_i \} _{i \in I}$ and $\{ M_l \} _{l \in L}$ be two families} of h-sets in $\mathbb{R}^n$ for which conditions~(B1) and~(B2) from Theorem~\ref{th:main} are satisfied. Then the set
\begin{displaymath}
  A = \left\{ z \in \mathbb{R}^n \Bigm| f^{k}(z) \in \correction{\left( \bigcup N_i\cup \bigcup M_l\right)} \text{ for all} \ k \in \mathbb{N} \right\}
\end{displaymath}
is a wall for $\{ N_i \} _{i \in I}$\correction{ and $\{ M_l \} _{l \in L}$}.
\end{lemma}

\begin{proof}
By definition, $\bigcup N_i\correction{\, \cup \, \bigcup M_l}$ is trivially a wall in $\{ N_i\}_{i \in I}$ \correction{and $\{ M_l \} _{l \in L}$}. We now apply Lemma~\ref{lem:wall-propagation} inductively to obtain a sequence of walls $A_\ell$ for $\ell \in \mathbb{N}$, with
\begin{align*}
    A_0 &:= \bigcup N_i, \quad \mbox{and} \\
    A_{\ell+1} &:= \mathcal{G}\left( A_\ell \right).
  \end{align*}
Since $I$ \correction{and $L$ are} finite, the union $\bigcup N_i\correction{\, \cup \, \bigcup M_l}$ is compact and the limit $A_\infty = {\displaystyle \lim_{\ell \to \infty} A_\ell}$ exists and is not empty. More precisely, every sequence $\{ z_\ell \}_{\ell \in [0, \infty)}$ with $z_\ell \in A_\ell$ for all $\ell \in [0, \infty)$ has a convergent subsequence
\begin{displaymath}
  \{ z_{\ell_m} \}_{m = 0}^\infty \ \text{with} \ z_{\ell_m} \in A_{\ell_m}  \ \text{and} \
  \lim_{m \to \infty} z_{\ell_m} = z_\infty \in A_\infty.
\end{displaymath}
We proceed by showing that $A_{\infty}$ is a wall. Take a horizontal disk $[h]$ in \correction{$\{ N_i\}$ or $\{ M_l\}$}. Since the $A_\ell$ are walls, we have a sequence of points
\begin{displaymath}
  z_\ell \in [h] \cap A_\ell \subset \bigcup N_i\correction{\, \cup \, \bigcup M_l}.
\end{displaymath}
Again, the sets $\bigcup N_i\correction{\, \cup \, \bigcup M_l}$ and $[h]$ are compact, so we can choose a convergent subsequence $z_{\ell_m}$ and its limit ${\displaystyle \lim_{m \to \infty} z_{\ell_m}} \in [h] \cap A_\infty$, as required.

Finally, by construction, for every point $z \in A_\infty$ we have
\begin{displaymath}
  f^k(z) \in \bigcup N_i \correction{\, \cup \, \bigcup M_l} \quad \text{for every} \ k \geq 0,
\end{displaymath}
which means that $A_\infty \subset A$ and, hence, $A$ is a wall as well.
\end{proof}

We now proceed with the actual proof of Theorem~\ref{th:main}. Recall that we assume that there exist \correction{ two finite families of h-sets $\{ N_i \}_{i \in I} \subset U$ and  $\{ M_l \}_{l \in L} \subset U$} for a diffeomorphism $f : \mathbb{R}^n \to \mathbb{R}^n$, with $n \geq 3$, such that conditions~\emph{(B1)} and~\emph{(B2)} hold with $(d_x,\, d_y, \,d_u,\, d_s)$ and $d_y > d_s$.

\begin{proof}[Proof of Theorem \ref{th:main}]
The goal is to show that the transitive hyperbolic invariant set $\Lambda=\mathrm{inv}(f, U)$ is a $d_y$-blender.
Clearly
\begin{align}
  A & = \left\{ z \in \mathbb{R}^n \Bigm| f^{k}(z) \in \bigcup N_i \correction{\, \cup \, \bigcup M_l}\ \text{for all} \ k \in \mathbb{N} \right\} \label{eq:A-in-Ws}   \\
  & \subset  \left\{ z \in \mathbb{R}^n \Bigm| f^{k}(z) \in U\ \text{for all} \ k \in \mathbb{N} \right\}  \subset \correction{W^s_{\mathrm{loc}}(\Lambda)}. \nonumber
\end{align}
Hence \correction{$W^s_{\mathrm{loc}}(\Lambda)$} is a wall, because $A$ is already a wall, by Lemma~\ref{lem:Wall}. This means that for every horizontal disk $[h]$ in $\{ N_i \}_{i \in I}$ \correction{or $\{M_l\}$} we have
\begin{displaymath}
  [h] \cap \correction{W_{\mathrm{loc}}^s(\Lambda)}\neq \emptyset.
\end{displaymath}
In particular, this is also true for all $C^1$ horizontal disks. Therefore, we can take the set of surfaces $\mathcal{S}$ from Definition~\ref{def:blender} to be the family of all $C^1$ horizontal disks $[h]$ in $\{ N_i \}_{i \in I}$. Horizontal disks are $d_y$-dimensional, and $d_y > d_s$, so it follows that $\Lambda$ is a $d_y$-blender, which concludes our proof.
\end{proof}

\begin{remark}
\label{rem:inv-non-empty}
Conditions~(B1) and~(B2) ensure that $\Lambda$ is non-empty; this is guaranteed by~\eqref{eq:A-in-Ws}, i.e., the fact that $A \subset \correction{W^s_{\mathrm{loc}}(\Lambda)}$, and $A \neq \emptyset$.
\end{remark}

\correction{Theorem \ref{th:main} can be reformulated to avoid the use of mother sets.}
\correction{\begin{theorem}
\label{cor:main}
Let $f : \mathbb{R}^n \to \mathbb{R}^n$ be a diffeomorphism, with $n \geq 3$, and assume that there exists a set $U\subset \mathbb{R}^n$ such that its maximal invariant set $\Lambda$ is hyperbolic and transitive.
Consider a finite family of h-sets $\{ N_i \}_{i \in I}\subset U$, for which the following condition is satisfied:
\begin{itemize}
\item[(B)] For every $l \in I$ and every horizontal disk $[h]$ in $N_l$, there exist $i,j \in I$ such that $[h] \cap N_i$  is a horizontal disk in $N_i$, $N_i \cover{f} N_j$, and $f$ satisfies cone conditions from $N_i$ to $N_j$.
\end{itemize}
If, in addition, the dimension of the topological entry of the h-sets is larger than the dimension of the stable bundle of $\Lambda$, then $\Lambda$ is a blender.
\end{theorem}}
\correction{\begin{proof} Consider two finite families of h-sets $\mathcal{M}$ and $\mathcal{N}$ chosen as
\begin{align*}
	\mathcal{M} &= \left\{N_l \in \{N_i\} \bigm| \mbox{ the set } N_l\mbox{ does not cover any set from the family } \{N_i\}  \right\},\\
	\mathcal{N} &= \{N_i\} \setminus \mathcal{M}.
\end{align*}
From {\em (B)} it follows that the two families $\mathcal{N}$ and $\mathcal{M}$ satisfy {\em (B1)} and {\em (B2)}, with $\mathcal{M}$ playing the role of the family of mother sets. Hence, the claim follows from Theorem \ref{th:main}.
\end{proof}

This proof shows that mother sets arise naturally as part of the covering, which is why we find it intuitive to work with them in the first place in the formulation of Theorem \ref{th:main}.}

\subsection{Finding blenders in practice}
\label{sec:compsetup}
Notice that conditions~\emph{(B1)} and~\emph{(B2)} are formulated in terms of finite \correction{families of sets $\{ N_i \}_{i \in I},\{M_l\}_{l\in L} \subset \mathbb{R}^n$, and} it suffices to construct finite sequences of coverings
\begin{equation}
\label{eq:connecting-sequence}
  N_{i_0} \cover{f} N_{i_1} \cover{f} \ldots \cover{f} N_{i_k}=M_{l^\prime},
\end{equation}
for some $l^\prime \in L$, \correction{for~\emph{(B1)} and~\emph{(B2)} to hold}. Considering an arbitrary horizontal disk $[h]$ of some mother set $M_l$, we take $N_{i_0} \subseteq M_l$ as the initial h-set that intersects with $[h]$ and construct a covering sequence that satisfies cone conditions. The length and the choice of the h-sets involved in this covering sequence and the choice of $l^\prime \in L$ generally depends on the choice of $M_l$ and the choice of $[h]$. Intuitively, a sequence of the form~\eqref{eq:connecting-sequence} allows us to establish a `link' between a horizontal disk $[h]$ in a mother set $M_l$ and a possibly different mother set $M_{l^\prime}$.  For each such sequence, the set $M_{l^\prime}$ is placed at the end. Such a sequence can be constructed as a covering sequence of h-sets with cone conditions precisely if conditions~\emph{(B1)} and~\emph{(B2)} hold for some  $(d_x,\, d_y, \,d_u,\, d_s)$ with $d_y > d_s$.

\begin{figure}[t!]
  \centering
  \includegraphics{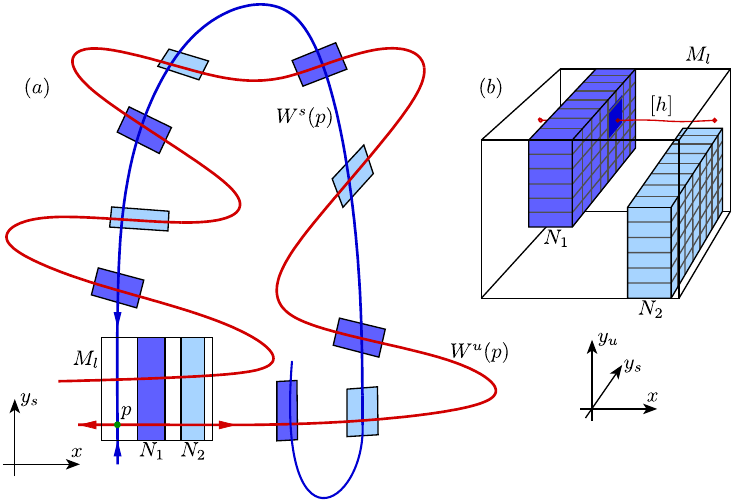}
  \caption{\label{fig:coverblender}
    Construction of a covering sequence. Panel~(a) shows an example where an  h-set $M_l$ near a hyperbolic fixed point is covered by an appropriate iterate of an h-set $N_1 \subset M_l$. As panel~(b) illustrates, the h-set $N_1$ itself is divided into overlapping smaller h-sets; such a subset is shown in darker blue with a horizontal disk $[h]$ through it.}
\end{figure}

\begin{remark}
\label{rem:sequences}
In fact, conditions~(B1) and~(B2) ensure that we are able to find sequences of the form~\eqref{eq:connecting-sequence}. By~(B1) we can always start from some h-set $N_{i_0}$ that, in a good way, intersects a horizontal disk $[h]$ from a mother set $M_l$. Condition~(B2) ensures for any given h-set $N_i$ $f$-covers some other h-set. Hence, a covering starting from $N_{i_0}$ either $f$-covers some $M_{l^{\prime}}$ in one step, or it can be continued to a longer sequence of coverings of the form~\eqref{eq:connecting-sequence}. Condition~(B2) also ensures that each such sequence ends with a mother set $M_{l^\prime}$ for some $l^\prime \in L$. (We cannot have infinite sequences since for each covering, the sets need to be enlarged along the $y_u$-coordinate; an infinite sequence of coverings would lead to a blow-up.) Thus, we start in a horizontal disk in a mother set $M_l$ and finish in this or another mother set $M_{l^\prime}$.
\end{remark}

As an example, Figure~\ref{fig:coverblender} illustrates how such a covering can be found near a homoclinic orbit to a hyperbolic fixed point $p$. The mother set $M_l$ in panel~(a) is chosen such that it contains $p$ and part of its stable and unstable manifolds, denoted $W^u(p)$ and $W^s(p)$, respectively. Near a homoclinic orbit, the sequence of h-sets $N_i$, with $N_{i_0} \in \{N_1,N_2\} \subset M_l$, can be chosen such that, after a sequence of coverings, the h-set $N_{i_k}$, for some $i_k \in I$,  returns to $M_l$; hence, $M_{l^\prime} = M_l$ in this example. Figure~\ref{fig:coverblender}(a) shows two sequences of h-sets:
one covering starts with $N_{i_0} = N_1$ (dark blue) and the other covering starts with $N_{i_0} = N_2$ (light blue). The h-sets are positioned along the intersections of the stable and unstable manifold of $p$. Successive h-sets $N_{i_m}$ in the covering sequence must be increasingly elongated along their corresponding $y_u$-coordinates, which also means that the final mother set $M_{l^\prime}$ in the sequence is chosen to be large along this coordinate; consequently, it will not cover any of the sets $N_{i_k}$ since these are smaller along $y_u$.

In practice, to establish cone conditions as required for condition~\emph{(B2)}, it is convenient to choose small initial h-sets $N_{i_0}$, as in Figure~\ref{fig:coverblender}(b), instead of the large initial h-sets like $N_1$ and $N_2$ in panel~(a). These initial h-sets need to satisfy condition~\emph{(B1)}. Thus, they need to overlap to ensure that any horizontal disk $[h]$ from $M_l$ leads to a horizontal disk in some initial $N_{i_0}$. Choosing small initial h-sets allows to use smaller intermediate h-sets in the sequences of coverings, which improves the accuracy of computations.

\section{\correction{$C^1$-persistent} heterodimensional cycles}
\label{sec:heteroclinic}
We foresee that our method can be particularly useful to locate and compute (\correction{$C^1$-persistent}) heterodimensional cycles, which are key elements for generating higher-dimen-sional (wild) chaotic dynamics~\cite{bonatti2005, LiTuraev2017, SaikiTakahasiYorke2021}. 

As a motivating example of a simple heterodimensional cycle in $\mathbb{R}^{3}$, we consider a $2$-blender $\Lambda$, with $d_s = 1$ and $d_u = 2$, and a hyperbolic fixed point $\widetilde{\Lambda} = p$ with $\widetilde{d_s} = 2$ and $\widetilde{d_u} = 1$. Hence, $p$ has a two-dimensional stable manifold $W^s(p)$ and a one-dimensional unstable manifold $W^u(p)$, which interact with the invariant manifolds of the blender $\Lambda$. %
%
\begin{figure}[t!]
  \centering
  \includegraphics{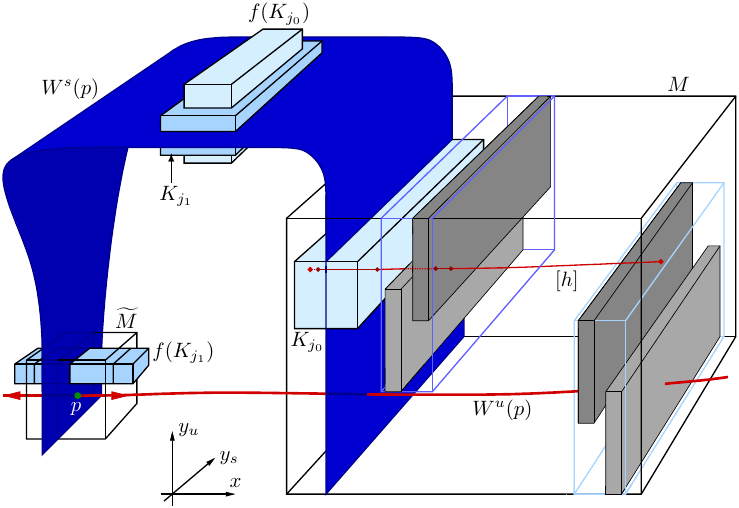}
  \caption{\label{fig:blender-point}
    Sketch of a heterodimensional cycle between a $2$-blender in $M$ and a hyperbolic fixed point $p$ in $\widetilde{M}$, featuring a connecting sequence $K_{j_0}\cover{f} K_{j_1} \cover{f} \widetilde{M}$ from $M$ to $\widetilde{M}$.}
\end{figure}
%
This situation is illustrated in Figure~\ref{fig:blender-point}, where the blender is the invariant set inside $M$ and the hyperbolic fixed point $p$ lies in $\widetilde{M}$. Figure~\ref{fig:blender-point} also shows an example of a connecting sequence of covering relations, that we require for our construction.


\correction{The above simple example is just a motivation. In general, we consider three finite families of h-sets $\{M_l\}_{l\in L}$, $\{\widetilde{M}_i\}_{i\in I}$ and $\{K_{j}\}_{j\in J}$. We assume that all of the h-sets in these three families have the same topological exit dimension $d_x$ and topological entry dimension $d_y$.}

\correction{We introduce the following definition of a connection between two families of h-sets, which is a building block for condition {\em(B3)} in Theorem \ref{th:cycles}.} 

\correction{\begin{definition}\label{def:connected}We say that the family of h-sets $\{M_l\}$ is connected to the family of h-sets $\{\widetilde{M}_i\}$ by sequences of covering relations if for every $l \in L$ and for every horizontal disk $[h]$ in $M_l$ there exists a sequence of h-sets $K_{j_0}, K_{j_1},\ldots,K_{j_r}$ and a set $\widetilde{M}_i$ (which may depend on $[h]$ and $l$) such that

\begin{enumerate}
\item $[h] \cap K_{j_0} = [h_{j_0}]$ is a horizontal disk $[h_{j_0}]$ in $K_{j_0}$;
\item There is a sequence of coverings

\begin{equation}
  K_{j_0} \cover{f} K_{j_1} \cover{f} \ldots \cover{f} K_{j_r} \cover{f} \widetilde{M}_i; \label{eq:connecting-seq-new}
\end{equation}

\item The diffeomorphism $f$ satisfies cone conditions from $K_{j_m-1}$ to $K_{j_m}$, for $m =1, \ldots, r$, as well as cone conditions from $K_{j_r}$ to $\widetilde{M}_i$.

\item The length $r$ of the sequence of coverings \eqref{eq:connecting-seq-new} is bounded from above by a constant that is independent of the choice of $[h]$ and $l$.

\end{enumerate}
\end{definition}}

\correction{Figure~\ref{fig:blender-point} shows an example of a sequence of covering relations
$K_{j_0}\cover{f} K_{j_1} \cover{f} \widetilde{M}$ needed for establishing that $M$ is connected to $\widetilde{M}$, for the case of a 2-blender in $M$ and a hyperbolic fixed point in $\widetilde{M}$.}

\correction{\begin{definition}\label{def:selfconnected}
We say that a family of h-sets $\{M_l\}\subset U$ is connected to itself in $U$, if $\{M_l\}$ is connected to itself by sequences of covering relations, and all the h-sets $K_{j_m}$ involved in the sequences of coverings  \eqref{eq:connecting-seq-new} are subsets of $U$.
\end{definition}
}
\correction{By Remark~\ref{rem:sequences} we see that if $\{M_l\}\subset U$ is a family of mother sets with an associated family of sets $\{N_i\}\subset U$ for which conditions {\em(B1)} and {\em(B2)} are fulfilled, then $\{M_l\}$ is connected to itself in $U$. (The h-sets $N_i$ then play the role of the h-sets $K_j$.)}

\correction{We now introduce the notion that two families of h-sets are {\em interconnected}, which is used to formulate condition {\em(B3)}.}
\correction{\begin{definition}\label{def:interconnected}
We say that $\{M_l\}\subset U$ and $\{\widetilde{M}_i\}\subset \widetilde{U}$ are interconnected by sequences of covering relations if:
\begin{enumerate}
\item $\{M_l\}$ is connected to $\{\widetilde{M}_i\}$ by sequences of covering relations;
\item $\{\widetilde{M}_i\}$ is connected to $\{M_l\}$ by sequences of covering relations;
\item $\{M_l\}$ is connected to itself in $U$;
\item $\{\widetilde{M}_i\}$ is connected to itself in $\widetilde{U}$.
\end{enumerate}
\end{definition}}

\correction{
Note that the sets $K_{j_m}$ involved in the covering relations \eqref{eq:connecting-seq-new} for conditions~1 and~2 of Definition~\ref{def:interconnected} do not need to be contained in $U\cup \widetilde{U}$. In conditions~3 and~4 of Definition~\ref{def:interconnected}, on the other hand, we require the h-sets involved in the sequences of covering relations to be contained in $U$ and $\widetilde{U}$, respectively.
}

\correction{We now give the proof of Theorem \ref{th:cycles}.}

\begin{proof}[Proof of Theorem \ref{th:cycles}]
We only prove that $W^u(\Lambda) \cap W^s(\widetilde{\Lambda}) \neq \emptyset$; the proof of $W_{\Lambda }^s(\Lambda) \cap W^u(\widetilde{\Lambda})$ follows from mirror arguments.

\correction{From condition~3 of Definition \ref{def:interconnected} it follows that} for every $l \in L$ and every horizontal disk $[h]$ in $M_l$ there exists a finite sequence of h-sets \correction{$\{ K_{j_m} \}\subset U$ and an h-set $M_{l^{\prime }}\subset U$ such that $[h]\cap K_{j_0}$ is a horizontal disk in $K_{j_0}$ and}
\begin{equation}
  M_l \supset \correction{K_{j_0}} \cover{f} \correction{K_{j_1}} \cover{f} \ldots \cover{f} \correction{K_{j_{k[h]}}} \cover{f} M_{l^{\prime }}.
  \label{eq:cycle-prf-tmp-1}
\end{equation}
The choice of the h-sets $\correction{K_{j_0}, K_{j_1}, \ldots,\, K_{j_{k[h]}}}$ and the choice of $k[h] \geq 1$ may depend on $[h]$. \correction{It is important to note that the h-sets involved in \eqref{eq:cycle-prf-tmp-1} are all contained in $U$.} By Theorem~\ref{th:disk-propagation} there exists a horizontal disk $[h^{\prime }] \in M_{l^{\prime }}$ such that
\begin{displaymath}
  [h^{\prime }] = \left\{  f^{k[h]}(z) \bigm|  z \in [h] \ \text{and} \ f^m(z) \in \correction{K_{j_m}\subset U} \ \text{for} \ m = 1, \ldots, k[h] \right\}.
\end{displaymath}
We wish to emphasise the dependence of the horizontal disk $[h^{\prime}] \in M_{l^{\prime}}$ on $[h] \in M_l$ and introduce the notation
\begin{displaymath}
  \mathcal{N}([h]) := [h^{\prime}].
\end{displaymath}
Hence, for $z \in \mathcal{N}([h])$ we have $f^{-m}(z) \in \correction{U}$ for $m = 1, \ldots, k[h]$, for some $k[h] \geq 1$.

Similarly, condition 4. from Definition \ref{def:interconnected} and Theorem~\ref{th:disk-propagation} imply that for every $i \in I$ and every horizontal disk $[\widetilde{h}] \in \widetilde{M}_i$, there exists $i^{\prime} \in I$ and a horizontal disk $\mathcal{\widetilde{N}}(\widetilde{h}) \in \widetilde{M}_{i^{\prime}}$, constructed from $k[\widetilde{h}]$ forward images of $[\widetilde{h}]$. \correction{Moreover, for every $z\in [\widetilde{h}]$ such that $f^{k[\widetilde{h}]}\in \mathcal{\widetilde{N}}(\widetilde{h})$ we have $f^m(z)\in \widetilde{U}$ for $m = 1, \ldots, k[\widetilde{h}]$.}

Finally, condition 1. from Definition \ref{def:interconnected} and Theorem~\ref{th:disk-propagation} imply that for every $l \in L$ and every horizontal disk $[h]$ in $M_l $ there exists $i \in I$ and a finite connecting sequence
\begin{equation}
  M_l \supset K_{j_0} \cover{f} K_{j_1} \cover{f} \ldots \cover{f} K_{j_{r[h]}} \cover{f} \widetilde{M}_{i},\label{eq:connecting-seq-new-2}
\end{equation}
with $r[h] \geq 1$, such that a horizontal disk $\mathcal{K}([h])$ in $\widetilde{M}_{i}$ can be constructed from forward images of $[h]$.

We now take an arbitrary $l \in L$ and an arbitrary horizontal disk $[h_0]$ in $M_l$, \correction{and construct the triple of horizontal disks}
\begin{displaymath}
  [h_{\ell}] :=\mathcal{N}^{(\ell)}([h_0]), \quad [\widetilde{h}_{\ell}] := \mathcal{K}([h_{\ell}]), \quad
  [\widetilde{h}_{\ell}^{\prime }] := \mathcal{\widetilde{N}}^{(\ell)}([\widetilde{h}_{\ell}]), \qquad
  \text{for } \ell \in \mathbb{N},
\end{displaymath}
by iterating $\mathcal{N}(\cdot)$ in and $\mathcal{\widetilde{N}}(\cdot)$ an arbitrary $\ell \in \mathbb{N}$ times. Consequently, we obtain a point $\widetilde{z}_{\ell}^{\prime} \in [\widetilde{h}_{\ell}^{\prime}]$ such that
\begin{enumerate}
\item
there exists $n_1(\ell) \geq \ell$ such that $\widetilde{z}_{\ell} := f^{-n_1(\ell)}(\widetilde{z}_{\ell}^{\prime }) \in [\widetilde{h}_{\ell}]$ and $f^{-m}(\widetilde{z}_{\ell}^{\prime }) \in \correction{\widetilde{U}}$ for $m = 0, \ldots, n_1(\ell)$;
\item
there exists $n_2(\ell) \geq 1$ such that $z_{\ell} := f^{-n_2(\ell)}(\widetilde{z}_{\ell}) \in [h_{\ell}]$;
\item
there exists $n_3(\ell) \geq \ell$ such that $f^{-n_3(\ell)}(z_{\ell}) \in [h_0]$ and $f^{-m}(z_{\ell}) \in \correction{U}$ for $m = 0, \ldots, n_3(\ell)$.
\end{enumerate}
Since the length of each connecting sequence \correction{\eqref{eq:connecting-seq-new-2}} is bounded by the same constant, the total number of possible coverings in the connecting sequences is finite. Therefore, we can choose $n_2 \geq 1$ fixed and select only a subsequence of triples $[h_{\ell_m}]$, $[\widetilde{h}_{\ell_m}]$ and $[\widetilde{h}_{\ell_m}^{\prime }]$ with $n_2(\ell_m) = n_2$, independently of $\ell_m$.

Recall that \correction{$\bigcup M_l$ and $\bigcup \widetilde{M}_i$} are compact, which implies that both sequences $z_{\ell_m}$ and  $\widetilde{z}_{\ell_m}$contain a convergent subsequence. Hence, we can choose a subsequence $\ell_{m_k}$ such that
\begin{displaymath}
	\lim_{k \to \infty} z_{\ell_{m_k}} = z^{\ast} \in  \correction{\bigcup M_l \subset U},
   		\quad \text{and} \quad
  	\lim_{k \to \infty} \widetilde{z}_{\ell_{m_k}} = \widetilde{z}^{\ast} \in  \correction{\bigcup \widetilde{M}_i \subset \widetilde{U}}
\end{displaymath}
where $f^{n_2}(z^{\ast}) = \widetilde{z}^{\ast}$. \correction{From the construction it follows that $f^{-k}(z^{\ast})\in U$ and $f^{k}(\widetilde{z}^{\ast})\in \widetilde{U}$ for every $k\ge 0$. Since $\Lambda=\mathrm{Inv}(f,U)$ and $f^{-k}(z^{\ast})\in U$ for all $k\ge 0$, we see that $z^{\ast}\in W^u(\Lambda )$. Similarly, since $\widetilde{\Lambda}=\mathrm{Inv}(f,\widetilde{U})$ and $f^{k}(\widetilde{z}^{\ast})\in \widetilde{U}$ for all $k\ge 0$, we see that $\widetilde{z}^{\ast}\in W^s(\widetilde{\Lambda})$. Moreover,  $f^{n_2}(z^{\ast}) = \widetilde{z}^{\ast}$,} hence, we have established the existence of a trajectory passing through $z^{\ast }$ that belongs to $W^u(\Lambda) \cap W^s(\widetilde{\Lambda})$. Therefore, $W^u(\Lambda)\cap W^s(\widetilde{\Lambda}) \neq \emptyset$, as required.
\end{proof}

Theorem~\ref{th:cycles} can be used, in particular, to establish a heterodimensional cycle between a blender $\Lambda$ and a hyperbolic fixed point $p$, as suggested in Figure~\ref{fig:blender-point}. The hyperbolic fixed point can be enclosed in a single h-set $\widetilde{M}$ with $\widetilde{M} \cover{f} \widetilde{M}$. Since for the blender $\Lambda$ we have $d_y > d_s$, the cycle is indeed heterodimensional.

\begin{remark}
Since the \correction{covering relations and cone conditions involved in condition {\em (B3)} of} Theorem~\ref{th:cycles} is \correction{persistent} under $C^1$-perturbations, our approach constitutes a new proof that the existence of a heterodimensional cycle is a $C^1$-\correction{persistent} property; see also~\cite{bonatti2012Nonlin,bonatti2005}.
\end{remark}

\section{A $2$-blender in the H{\'e}non-like family}
\label{sec:CAP}
We are finally ready to prove Theorem~\ref{th:henon}, where we claim that the H{\'e}non-like family~\eqref{eq:Henon-family} has a $2$-blender for every $\xi \in [ 1.01,\, 1.125]$ when $\mu = -9.5$ and $\beta = 0.3$. The proof precisely follows the approach described above. We construct explicit covering relations and verify the conditions of Theorem~\ref{th:main} with computer-assisted methods; we then apply Theorem~\ref{th:main}, which completes the proof. The constructed covering relations also provide an explicit bound on the set in phase space that contains the surfaces that intersect the stable manifold of the $2$-blender.

Recall that the H{\'e}non-like family is defined by the function $f : \mathbb{R}^3 \to \mathbb{R}^3$ with
\begin{displaymath}
  f\left(\mathrm{x},\, \mathrm{y},\, \mathrm{z}\right) =
  \left( \mathrm{y},\, \mu + \mathrm{y}^2 + \beta \, \mathrm{x}, \xi \, \mathrm{z}+\mathrm{y} \right),
\end{displaymath}
and we fix $\mu = -9.5$ and $\beta = 0.3$ as in Theorem~\ref{th:henon}. We assume $\xi > 1$, so that we have $d_x = 1$, and the skew-product structure of $f$ ensures that the weakly expanding local coordinate $y_u$ is the $\mathrm{z}$-direction, provided $\xi$ is sufficiently small. Here, we use phase variables $\mathrm{x}$, $\mathrm{y}$ and $\mathrm{z}$ to distinguish them from the local coordinates $x \in \mathbb{R}^{d_x}$ and $y \in \mathbb{R}^{d_y}$. As can easily be checked, $f$ has the two fixed points
\begin{displaymath}
  p^{\pm } = \left( \rho^{\pm },\, \rho^{\pm },\, \frac{\rho^{\pm }}{1-\xi }\right),
\end{displaymath}
with
\begin{displaymath}
  \rho ^{\pm} = \tfrac{1}{2} \, (1 - \beta) \pm \tfrac{1}{2} \, \sqrt{(1 - \beta)^2 - 4 \, \mu},
\end{displaymath}
which are hyperbolic saddles for $\mu = -9.5$, $\beta = 0.3$ and $\xi > 1$, with two-dimensional unstable and one-dimensional stable manifolds that intersect transversally. This follows from the fact that the restriction of $f$ to the $(\mathrm{x},\mathrm{y})$-plane is the standard H{\'e}non map, which is known to feature a full Smale horseshoe for $\mu = -9.5$ and $\beta =0.3$. Moreover, this skew-product structure of $f$ implies that the $\mathrm{z}$-direction is an $f$-invariant (weakly unstable) fibre bundle for points in $\mathbb{R}^3$. This means, in particular, that the projection of the one-dimensional stable manifold $W^s(p^+)$ of $p^+$ is exactly the stable manifold of the fixed point $(\rho^+,\, \rho^+)$ of the standard H{\'e}non map, while the two-dimensional unstable manifold $W^u(p^+)$ projects to a one-dimensional curve in the $(\mathrm{x},\mathrm{y})$-plane, namely, exactly the one-dimensional unstable manifold of $(\rho^+,\, \rho^+)$ for the standard H{\'e}non map; in other words, $W^u(p^+)$ is a ruled manifold given by the direct product of this one-dimensional unstable manifold and vertical straight lines. A top view of the fixed points $p^\pm$ with the manifolds $W^u(p^+)$ and $W^s(p^+)$ is shown in~Figure~\ref{fig:manifolds}; the manifolds of $p^-$ are not shown in this figure.
%
\begin{figure}[t!]
  \centering
  \includegraphics{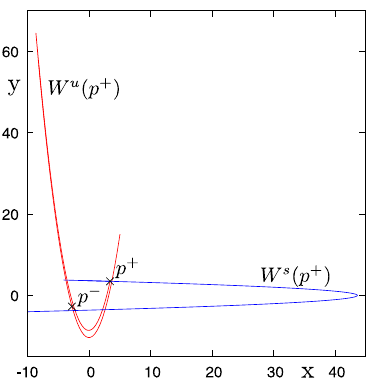}
  \caption{\label{fig:manifolds}
    Projection onto the $(\mathrm{x},\mathrm{y})$-plane of the fixed points $p^{\pm }$ (crosses) of the H{\'e}non-like map~\eqref{eq:Henon-family} with $\mu =-9.5$ and $\beta =0.3$, with computed initial parts of the unstable and stable manifolds of $p^+$, labeled $W^u(p^+)$ (red curve) and $W^s(p^+)$ (blue curve), respectively.}
\end{figure}

\begin{figure}[t!]
  \centering
  \includegraphics[width=5cm]{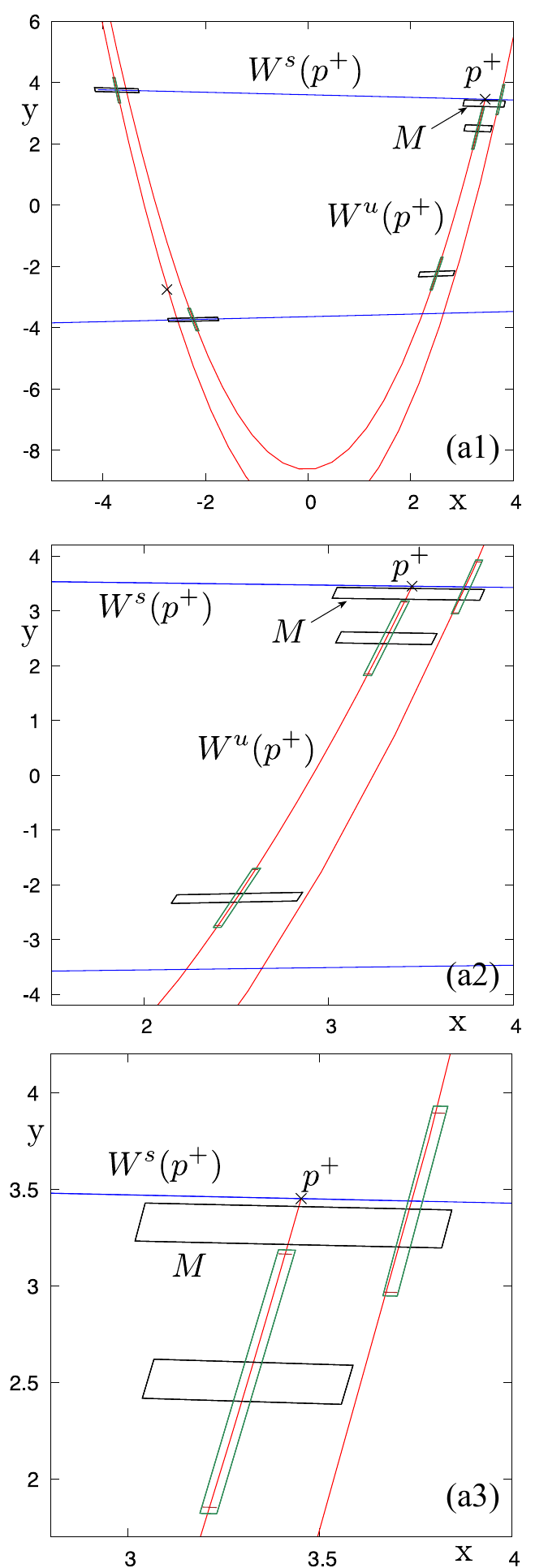}
  \includegraphics[width=5cm]{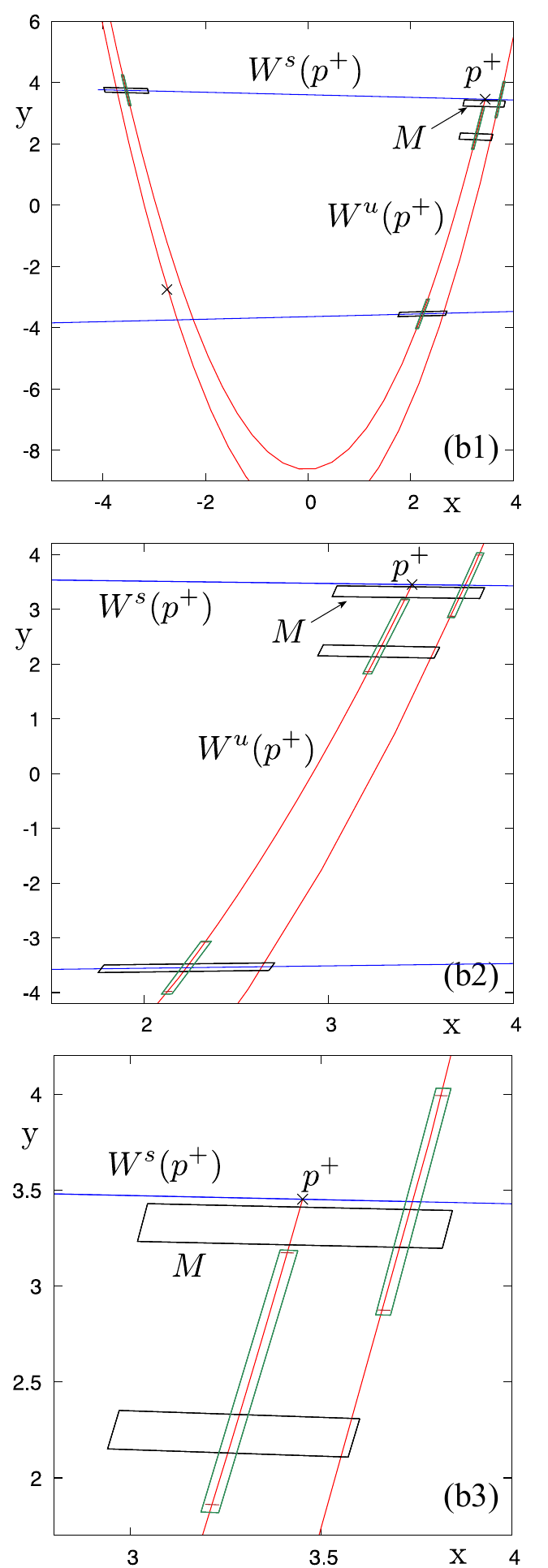}
  \caption{\label{fig:windows-xy}
    Top view of the two covering sequences used in the proof, shown by three successive enlargements of the $(\mathrm{x},\mathrm{y})$-plane in panels~(a1)--(a3) and~(b1)--(b3), respectively. The first sequence consists of five and the second of four h-sets (black `horizontal rectangular' sets) parallel to $W^s(p^+)$. The images of these h-sets (dark green `vertical rectangular' sets) are elongated along $W^u(p^+)$; notice the images of the exit sets (brown horizontal lines) in panels~(a3) and~(b3).}
\end{figure}

\subsection{Construction of a family of h-sets}
\label{sec:boxes}
Before we can apply Theorem~\ref{th:main}, we must construct a family of the h-sets $\{ N_i \}_{i \in I}$, with a selected family of mother sets $\{ M_l\}_{l \in L}$ such that conditions~\emph{(B1)} and~\emph{(B2)} are satisfied. To this end, it suffices to consider only two covering sequences, namely, h-sets that are aligned with the manifolds of $p^+$ and positioned along two different homoclinic orbits to $p^{+}$. Figure~\ref{fig:windows-xy} gives an idea of their specific locations in projection onto the $(\mathrm{x},  \mathrm{y})$-plane; compare also with Figure~\ref{fig:coverblender}. The first column of Figure~\ref{fig:windows-xy} shows four h-sets (black boxes) chosen such that they cover one of these homoclinic orbits; panel~(a1) shows five h-sets and the successive enlargements in panels~(a2) and~(a3) provide more detail near $p^+$. The second column of Figure~\ref{fig:windows-xy} shows the selection of four h-sets that cover another homoclinic orbit; again, panel~(b1) shows the complete covering sequence and panels~(b2) and (b3) are successive enlargements near $p^+$. The dark-green sets in Figure~\ref{fig:windows-xy} are the images of the nine h-sets and the horizontal brown lines in the bottom four panels indicate the images of the exit sets. Note that each rectangle is, in fact, a three-dimensional box with an appropriate range of the $\mathrm{z}$-coordinate. In particular, the images of the h-sets are also straight boxes with respect to the $\mathrm{z}$-coordinate, but the actual $\mathrm{z}$-range for these images, as well as that of their exit sets, depends on the $\mathrm{z}$-range of the corresponding h-set.

To be more precise, for both coverings, we use the same mother set, which is the h-set labeled $M$ that lies closest to $p^+$ in Figure~\ref{fig:windows-xy}; this box is positioned slightly away from $p^+$ along $W^u(p^+)$; it is best seen in the two bottom enlargements in panels~(a3) and~(b3) of Figure~\ref{fig:windows-xy}. The sides of the mother set, as shown in the $(\mathrm{x},  \mathrm{y})$-plane, are roughly parallel to the directions of strong expansion (or the direction of $W^u(p^+)$ in Figure~\ref{fig:windows-xy}), and contraction (the direction of $W^s(p^+)$). These h-sets are boxes, that is, they also have sides parallel to the $\mathrm{z}$-direction, but this is not shown in Figure~\ref{fig:windows-xy}. Following Definition~\ref{def:hset}, we wish to represent $M$ in local coordinates given by the splitting $(x,\, y) := (x,\, y_u,\, y_s)$ with $d_x = 1$ and $d_y = 2$, as the product $\overline{B}_{d_x} \times \overline{B}_{d_y}$ of two balls. As mentioned in Remark~\ref{rem:norms}, a natural choice is to use the maximum norm on both $\overline{B}_{d_x}$ and $\overline{B}_{d_y}$, so that $M$ is represented by a Cartesian product of intervals. We define the local coordinate transformation $\gamma_M : M \to \overline{B}_{d_x} \times \overline{B}_{d_y}$, such that
\begin{equation}
\label{eq:M-set}
  \gamma_M(M) := [-0.1, 0.1] \times \Big( [-2, 2] \times [-0.4, 0.4] \Big).
\end{equation}
Here, $\gamma_M$ is defined explicitly as the affine map
\begin{equation}
  \gamma _M(z) := A_M^{-1} \, [z - p_M(\xi)], \text{ for } z \in M, \label{eq:coordinate-change-M}
\end{equation}
where the columns of the matrix $A_M$ are formed by the (normalized) eigenvectors for $\xi = 1.1$, that is,
\begin{displaymath}
  A_M:= \left[ \begin{array}{r@{.}lcr@{.}l}
                   0 & 131936 & 0 & -0 & 998261 \\
                   0 & 984126 & 0 &   0 & 0447916 \\
                   0 & 118698 & 1 & -0 & 0383312
                 \end{array} \right],
\end{displaymath}
and the origin is shifted to the center of the box, which we choose as the point
\begin{displaymath}
  p_M = p_M(\xi) := \left( 3.4319,\, 3.4319,\, z_{M}(\xi) \right),
\end{displaymath}
close to $p^+$; here, the $\mathrm{x}$- and $\mathrm{y}$-components are explicit, but more work is needed to decide on the $\mathrm{z}$-component $z_{M}(\xi)$, which is explained further below. We remark that $A_M$ can be kept fixed over the entire parameter range $\xi \in [1.01, 1.125]$.

The other h-sets in the sequences of covering relations are chosen similarly in an iterative manner, with the $\mathrm{x}$- and $\mathrm{y}$-coordinates of their center points approximately positioned at the two homoclinic orbits; compare with Figure~\ref{fig:windows-xy}. Leaving the specific choice of $Z$-coordinate again for later, we select the following two sequences:
\begin{align*}
  p_{10}(\xi) & := p_M(\xi), \\
  p_{11}(\xi) & :=\left( \phantom{-}3.3127,\, \phantom{-}2.5032,\, z_{11}(\xi)\right), \\
  p_{12}(\xi) & :=\left( \phantom{-}2.5032,\, -2.2401,\, z_{12}(\xi)\right), \\
  p_{13}(\xi) & :=\left( -2.2401,\, -3.7312,\, z_{13}(\xi)\right), \\
  p_{14}(\xi) & :=\left( -3.7312,\, \phantom{-}3.7495,\, z_{14}(\xi)\right).
\end{align*}
and
\begin{align*}
  p_{20}(\xi) & := p_M(\xi), \\
  p_{21}(\xi) & :=\left( \phantom{-}3.2714,\, \phantom{-}2.2300,\, z_{21}(\xi)\right), \\
  p_{22}(\xi) & :=\left( \phantom{-}2.2300,\, -3.5459,\, z_{22}(\xi)\right), \\
  p_{23}(\xi) & :=\left( -3.5459,\, \phantom{-}3.7421,\, z_{23}(\xi)\right);
\end{align*}
Observe that, in projection onto the $(\mathrm{x}, \mathrm{y})$-plane, $p_{11} \mapsto p_{12} \mapsto p_{13} \mapsto p_{14}$ and $p_{21} \mapsto p_{22} \mapsto p_{23}$ under the action of $f$. Furthermore, to good approximation, we have $p_{10} \mapsto p_{11}$ and $p_{20} \mapsto p_{21}$, while the $(\mathrm{x}, \mathrm{y})$-projection of the images of both $p_{14}$  and $p_{23}$ is approximately that of $p_M$ as well. The linearization about each of the sequences $\{ p_{1b} \}_{b = 0, \ldots, k_1}$, with $k_1 = 4$, and $\{ p_{2b}\}_{b = 0, \ldots, k_2}$, with $k_2 = 3$, gives rise to appropriate affine changes of coordinates of the form
\begin{equation}
  \gamma_{1b}(z) = A_{1b}^{-1} \, [z - p_{1b}(\xi)] \qquad \mbox{and} \qquad
  \gamma_{2b}(z) = A_{2b}^{-1} \, [z - p_{2b}(\xi)], \label{eq:coordinate-changes-i}
\end{equation}
with $A_{10} = A_{20} = A_M$ and
\begin{displaymath}
  \begin{array}{rclcl}
    A_{1(b+1)} &\approx& Df(p_{1b}) \, A_{1b} & \mbox{for} & b = 0,\ldots, k_1,\\
    A_{2(b+1)} &\approx& Df(p_{2b}) \, A_{2b} & \mbox{for} & b = 0,\ldots, k_2,
  \end{array}
\end{displaymath}
that define the h-sets $N_{1b}$, for $b = 0, \ldots, k_1$ and $N_{2b}$, for $b = 0, \ldots, k_2$. Note that $\gamma_{10} = \gamma_{20} = \gamma_{M}$, since $A_{10} = A_{20} = A_M$ and $p_{10}(\xi) = p_{20}(\xi) = p_M(\xi)$. The local coordinates
of the other h-sets are defined recursively based on successive iterates of $f$. Specifically, the matrices for coordinate changes at points $p_{1b}$ along the first homoclinic orbit were chosen as
\begin{displaymath}
  \begin{array}{rclrcl}
    A_{11} &=& \left[ \begin{array}{r@{.}lrr@{.}l}
     \phantom{-}0 & 15187 & 0 & -0 & 64804 \\
                        1 & 0123     & 0 &   0 & 039354 \\
                        0 & 17202 & 1 & -0 & 038011
                      \end{array} \right], &
    A_{12} &=& \left[ \begin{array}{r@{.}lrr@{.}l}
                        0 & 15622 & 0 &   0 & 85005 \\
                        0 & 78914 & 0 &   0 & 056412 \\
                        0 & 18542   & 1 & -0 & 053097
                    \end{array} \right], \\[8mm]
    A_{13} &=& \left[ \begin{array}{r@{.}lrr@{.}l}
                          0 & 12178   & 0 &   1 & 2185 \\
                        -0 & 53836 & 0 &   0 & 04927 \\
                          0 & 15326 & 1 & -0 & 043093
                   \end{array} \right], &
    A_{14} &=& \left[ \begin{array}{r@{.}lrr@{.}l}
                        -0 & 08308     & 0 &   1 & 0643 \\
                          0 & 62561   & 0 & -0 & 046216 \\
                        -0 & 057064 & 1 &   0 & 040401
                   \end{array} \right];
  \end{array}
\end{displaymath}
and those for coordinate changes at points $p_{2b}$ along the second homoclinic orbit were chosen as
\begin{displaymath}
  \begin{array}{rclrcl}
    A_{21} &=& \left[ \begin{array}{r@{.}lrr@{.}l}
                        -0 & 15209 & 0 & -0 & 78781 \\
                        -1 & 0012   & 0 &   0 & 053503 \\
                        -0 & 17003 & 1 & -0 & 050450
                      \end{array} \right], &
    A_{22} &=& \left[ \begin{array}{r@{.}lrr@{.}l}
                        -0 & 15451 & 0 &   1 & 1557 \\
                        -0 & 69615 & 0 &   0 & 049190 \\
                        -0 & 18337 & 1 & -0 & 043018
                      \end{array} \right], \\[8mm]
    A_{23} &=& \left[ \begin{array}{r@{.}lrr@{.}l}
                        -0 & 10743   & 0 &   1 & 0625 \\
                          0 & 75471 & 0 & -0 & 046214 \\
                        -0 & 13856 & 1 &   0 & 040391
                      \end{array} \right].
  \end{array}
\end{displaymath}
The lengths of the column vectors for the matrices $A_{1b}$ and $A_{2b}$ are chosen so that the expansion and contraction in each local map is (roughly) the same for each covering relation involved in a given homoclinic excursion.
\begin{figure}[t!]
  \centering
  \includegraphics{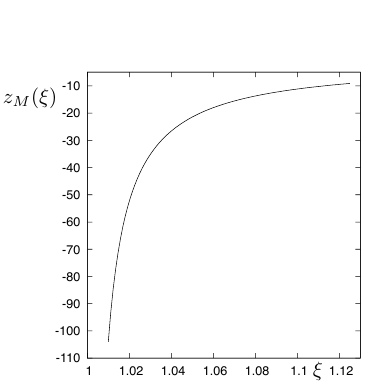}
  \caption{\label{fig:zM}
    Graph of $z_M(\xi)$ for $\xi \in [1.01,1.125]$ that satisfy Equations~\eqref{eq:xi-change-1} and~\eqref{eq:xi-change-2}.}
\end{figure}
%
Let us now return to the questions how to choose the $\mathrm{z}$-ranges for the h-sets $M$, $N_{1b}$, for $b = 0, \ldots, k_1$ and $N_{2b}$, for $b = 0, \ldots, k_2$. Recall that the $\mathrm{z}$-coordinate represents the weakly expanding coordinate $y_u$. Our goal is to achieve a covering relation as sketched in Figure~\ref{fig:blender_modified}. To this end, we define the $\mathrm{z}$-coordinates of the center points in each h-set $N_{1b}$ and $N_{2b}$ recursively as follows. We formally define
\begin{eqnarray*}
  z_{1(k+1)}(\xi) &=& \pi_{\mathrm{z}} f(p_{1k}(\xi)) = \xi \, z_{1k}(\xi) + \pi_{\mathrm{y}} p_{1k}(\xi),\\
  z_{2(k+1)}(\xi) &=& \pi_{\mathrm{z}} f(p_{2k}(\xi)) = \xi \, z_{2k}(\xi) + \pi_{\mathrm{y}} p_{2k}(\xi),
\end{eqnarray*}
where
\begin{displaymath}
  z_{10}(\xi) = z_{20}(\xi) := z_{M}(\xi),
\end{displaymath}
and $\pi_{\mathrm{y}}$ and $\pi_{\mathrm{z}}$ are the projections onto the $\mathrm{y}$- and $\mathrm{z}$-coordinates, respectively. In our implementation, we choose different values of $z_{M}(\xi)$ for different $\xi$ to ensure that, after an excursion along the first homoclinic orbit given by the points $p_{1b}$, the change in the $\mathrm{z}$-coordinate is negative, meaning that
\begin{equation}
\label{eq:xi-change-1}
  l_1(z_M(\xi),\, \xi) := \pi_{\mathrm{z}} f(p_{14}(\xi)) = \xi \, z_{14}(\xi) + \pi_{\mathrm{y}} p_{14}(\xi) < z_M(\xi).
\end{equation}
Simultaneously, we require that, after an excursion along the second homoclinic orbit given by $p_{2b}$, the change in the $\mathrm{z}$-coordinate is positive, meaning that
\begin{equation}
\label{eq:xi-change-2}
  l_2(z_M(\xi),\, \xi) := \pi_{\mathrm{z}} f(p_{23}(\xi)) = \xi \, z_{23}(\xi) + \pi_{\mathrm{y}} p_{23}(\xi) > z_M(\xi).
\end{equation}
By combining Equations~\eqref{eq:xi-change-1} and~\eqref{eq:xi-change-2}, we find $z_{M}(\xi) = z$ numerically as the $\xi$-parameterised solutions to the scalar equation $\frac{1}{2} ( l_1(z,\, \xi)+l_2(z,\, \xi )) - z = 0$. Figure~\ref{fig:zM} shows the result of this approach and represents the choice of $z_{M}(\xi)$ we used for the computer-assisted proof.

\begin{figure}[t!]
  \centering
  \includegraphics{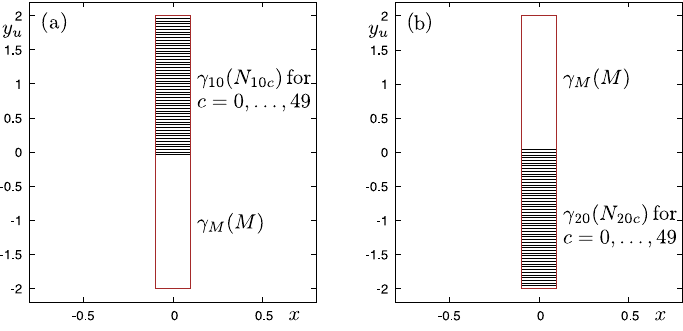}
  \caption{\label{fig:initial-hSets}
    Projection onto the local $(x, y_u)$-coordinate plane of the mother set $M$, together with, in panel~(a), the initial h-sets $N_{10c}$ along the first homoclinic orbit from Figure~\ref{fig:windows-xy}(a) and, in panel~(b), the initial h-sets $N_{20c}$ along the second homoclinic orbit from Figure~\ref{fig:windows-xy}(b), both for $c = 0, \ldots, 49$.}
\end{figure}
%
As mentioned just before the start of Section~\ref{sec:CAP}, to check condition~\emph{(B2)} in practice, it is convenient to subdivide the resulting two families of h-sets $\{ N_{1b} \}_{b = 0, \ldots, k_1}$ and $\{ N_{2b} \}_{b = 0, \ldots, k_2}$ into smaller overlapping h-sets; in local coordinates, this overlap only needs to be present in the $y_u$-direction. We use a total of $100$ subdivisions, $50$ each for $N_{1b}$ and $N_{2b}$, identified by the index $c \in \left\{ 0, 1, \ldots, 49 \right\}$. The entire family $\{ N_i \}_{i \in I}$ of h-sets is then defined as
\begin{displaymath}
  \gamma_{ab}^{-1}(N_{abc}) := [-0.1, 0.1] \times \Big( \mathcal{I}_{abc} \times [-0.4, 0.4] \Big),
\end{displaymath}
where the set of indices for the h-sets $N_i = N_{abc}$ is
\begin{displaymath}
  I := \left\{ abc \bigm| ab \in \left\{ 10, 11, 12, 13, 14, 20, 21, 22, 23\right\}, \mbox{and } c \in
\left\{ 0, 1, \ldots, 49 \right \} \right\} \cup \{ 0 \};
\end{displaymath}
here, the index $0$ is added so that we can choose $N_0 := M$. The intervals $\mathcal{I}_{abc}$ for the local coordinate $y_u$ are chosen as follows. The intervals $\mathcal{I}_{10c}$ and $\mathcal{I}_{20c}$ for $c \in \left\{ 0, 1, \ldots, 49 \right\}$ are defined as
\begin{displaymath}
  \mathcal{I}_{10c} := c \, \Delta + [-\Delta,\, \Delta] \qquad \mbox{and} \qquad
  \mathcal{I}_{20c} := -c \, \Delta + [-\Delta ,\, \Delta],
\end{displaymath}
with $\Delta = \frac{2}{50} = 0.04$. Note that
\begin{displaymath}
  \mathcal{I}_{100} = \mathcal{I}_{200} = [-\Delta,\, \Delta].
\end{displaymath}
Moreover, for $c = 49$, the right edge of $\mathcal{I}_{10c}$ is $2$ and the left edge of $I_{20c}$ is $-2$. Hence, the union of the $50$ intervals $I_{10c}$ form the interval $[-\Delta,\, 2]$ and the union of the $50$ intervals $I_{20c}$ forms the interval $[-2,\, \Delta]$. In total, we have $100$ overlapping intervals that range over $[-2, 2]$. Therefore, since $\gamma_M = \gamma_{10c} = \gamma_{20c}$, by construction, we have
\begin{equation}
\label{eq:M-parts}
  \gamma_M(M) =  \bigcup_{c=0}^{49} \gamma_{10}\left( N_{10c} \right) \cup \bigcup_{c=0}^{49} \gamma_{20}\left( N_{20c} \right).
\end{equation}
The h-sets $N_{10c}$ and $N_{20c}$, for $c \in \{ 0, \ldots, 49\}$ are shown in Figure~\ref{fig:initial-hSets} in projection onto the local $(x, y_u)$-coordinate plane. They are the initial h-sets for the sequences of coverings~\eqref{eq:connecting-sequence}. We intentionally arrange these h-sets to overlap so that they enclose every horizontal disk in $M$, as is required for condition~\emph{(B1)}; this is discussed in more detail at the end of Section~\ref{sec:validation-cones}.

The remainder of the intervals $\mathcal{I}_{abc}$, for $abc \in I$ with $b > 0$, are chosen recursively by computing the images under $\gamma_{ab} \circ f \circ \gamma_{a(b-1)}^{-1}$ of the boxes $[-0.1, 0.1] \times \Big( \mathcal{I}_{a(b-1)c} \times [-0.4, 0.4] \Big)$, and choosing local coordinates for $\mathcal{I}_{abc}$ such that it contains the $y_u$-range of the resulting image. This is done automatically by our implementation and ensures that we have topological entry along $y_u$.

\subsection{Validation of conditions (B1) and (B2)}
\label{sec:validation}
We claim that the family of h-sets $\{ N_{abc} \}_{abc \in I}$ \correction{and $M$}, with $ab \in \{  10, 11, 12, 13, 14, 20,$ $ 21, 22, 23 \}$ and $c \in \{ 0, \ldots, 49\}$ satisfies conditions~\emph{(B1)} and~\emph{(B2)} of Theorem~\ref{th:main}. In this section, we explain how we prove this in a computer-assisted manner. The additional requirement of hyperbolicity and transitivity of the invariant set in a suitable set $U\subset \mathbb{R}^3$ is addressed in the Appendix. There are various methods for the validation of covering relations and cone conditions; we direct the reader to~\cite{CG-CPAM,  MR2060532, MR2736320, MR2276430, MR2494688, MR2060531}, which deal with this topic in various contexts. When the topological exit set for the coverings and cone conditions is of dimension $d_x = 1$, which is the specific case for the H{\'e}non family, we can use the approach outlined in the following sections.

\subsubsection{Validating covering relations when $d_x = 1$}
\label{sec:validation-cover}
If $d_x = 1$ then $\overline{B}_{d_x} = [-1, 1]$ and the exit set $N^{-}_i$ consists of two parts for any $i \in I$; in the coordinates given by the homeomorphism $\gamma_i$ they are
\begin{displaymath}
  N_{\gamma_i,l}^{-} = \left\{ -1 \right\} \times \overline{B}_{d_y} \quad \mbox{and} \quad
  N_{\gamma_i,r}^{-} = \left\{ 1 \right\} \times \overline{B}_{d_y},
\end{displaymath}
where the subscript `$l$' stands for `left', and `$r$' for `right'. This notation allows us to formulate the following useful result.


\begin{lemma}
\label{lem:cov-1d-unstble}
Let $N_i, N_j \subset \mathbb{R}^n$, with $n \geq 3$, be two h-sets with associated transformations $\gamma_i$ and $\gamma_j$ that each map to a product of closed balls in local coordinates $(x, y)$; here, the component $x$ is one dimensional with corresponding projection operator $\pi_x$. Furthermore, let $f : \mathbb{R}^n \to \mathbb{R}^n$ be a diffeomorphism. Recall that the action of $f$ in local coordinates is denoted by $f_{ji} = \gamma_j \circ f \circ \gamma_i^{-1}$. Assume that either
\begin{equation}
\label{eq:u=1-covering-w=1}
  \pi _{x} \left[ f_{ji}(N_{\gamma_i,l}^{-}) \right] < -1 \quad \text{and} \quad \pi_{x} \left[ f_{ji}(N_{\gamma_i,r}^{-})  \right] > 1,
\end{equation}
or
\begin{equation}
\label{eq:u=1-covering-w=-1}
  \pi_{x} \left[ f_{ji}(N_{\gamma_i,l}^{-})  \right] > 1 \quad \text{and} \quad \pi_{x} \left[ f_{ji}(N_{\gamma_i,r}^{-})  \right] < -1.
\end{equation}
Then $N_i \cover{f} N_j$, provided
\begin{equation}
\label{eq:u=1-covering-contraction}
  f_{ji}(N_{\gamma_i}) \cap \overline{B}_{d_x} \times \left( \mathbb{R}^{d_y} \setminus B_{d_y}\right) = \emptyset.
\end{equation}
\end{lemma}
\begin{proof}
We define a homotopy $\varsigma : [0, 1] \times N_{\gamma_i} \to \mathbb{R}^{d_x} \times \mathbb{R}^{d_y}$ as
\begin{displaymath}
  \varsigma \left( t,\, (x, y) \right) = \left\{
    \begin{array}{lrl}
      \left( \pi _{x} f_{ji}(x, y),\, (1 - 2 t) \, \pi_{y} f_{ji}(x, y) \right), & \text{if} & t \in \left[ 0,\tfrac{1}{2} \right], \\[2mm]
      \left( (2 t - 1) \, A x + 2 (1 - t) \, \pi_{x} f_{ji}(x, y),\, 0\right), & \text{if} & t \in \left(\frac{1}{2},1\right],
    \end{array} \right.
\end{displaymath}
where $A$ is the trivial $1 \times 1$ matrix $A = [2]$ for the case of assumption~\eqref{eq:u=1-covering-w=1}, and $A = [-2]$ for assumption~\eqref{eq:u=1-covering-w=-1}. This homotopy fulfils the conditions from Definition~\ref{def:covering}.
\end{proof}

We use Lemma \ref{lem:cov-1d-unstble} to check whether the $c$-dependent families of sequences $\{ N_{abc} \}_{abc \in I}$ for $a = 0$ and $a = 1$ are, in fact, covering sequences. More precisely, for each $c \in \{ 0, \ldots, 49\}$, we validate that either assumptions~\eqref{eq:u=1-covering-w=1} and~\eqref{eq:u=1-covering-contraction} or assumptions~\eqref{eq:u=1-covering-w=-1} and~\eqref{eq:u=1-covering-contraction} hold for the h-sets
\begin{eqnarray}
\label{eq:coverings-1}
  N_{1(b-1)c} &\cover{f}& N_{1bc} \qquad \text{for } b = 1, 2, 3, 4, \\
\label{eq:coverings-2}
  N_{2(b-1)c} &\cover{f}& N_{2bc} \qquad \text{for } b = 1, 2, 3,
\end{eqnarray}
We also validate that the final pairs in these two sequences are covering relations, namely, that we have
\begin{eqnarray}
\label{eq:coverings-M-1}
  N_{14c} &\cover{f}& M \qquad \text{for } c = 0, \ldots, 49, \\
\label{eq:coverings-M-2}
  N_{23c} &\cover{f}& M \qquad \text{for } c = 0, \ldots, 49.
\end{eqnarray}
%
\begin{figure}[t!]
  \centering
  \includegraphics{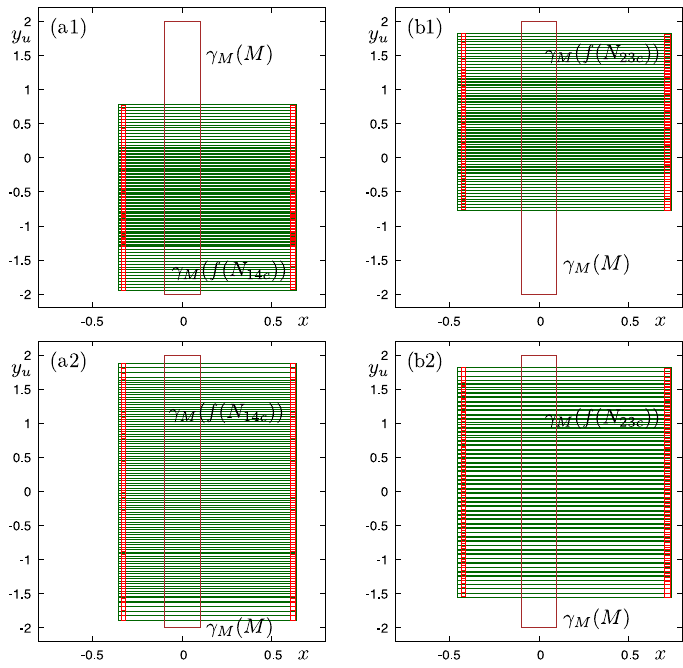}
  \caption{\label{fig:final-covering}
    The covering in local coordinates of the set $M$ (dark red vertical rectangle) by final sets, shown in projection onto the $(x, y_u)$-plane. Panels~(a1) and~(b1) show the coverings $N_{14c} \cover{f} M$ and $N_{23c} \cover{f} M$, for $\xi \in [1.01, 1.011]$, respectively, and panels~(a2) and~(b2) show these coverings for $\xi \in [1.124, 1.125]$; here $c = 0, \ldots, 49$ and the bounds on the images of the exit sets are shown in red.}
\end{figure}

We refer again to Figure~\ref{fig:windows-xy}, which also shows the images of the sets involved in the two covering sequences; those associated with the covering relations in~\eqref{eq:coverings-1} and~\eqref{eq:coverings-M-1} are depicted in the left column, and those associated with~\eqref{eq:coverings-2} and~\eqref{eq:coverings-M-2} in the right column. Figure~\ref{fig:final-covering} shows for each of the two homoclinic excursions, the last covering relation in the sequence in projection onto the $(x,y_u)$-plane of local coordinates; see also the green vertical boxes $f\left( N_{14c}\right)$ and $f\left( N_{23c}\right)$ in the top-right corners of panels~(a3) and~(b3) of Figure~\ref{fig:windows-xy}, respectively. In Figure \ref{fig:final-covering}, we also show the projections of the coverings~\eqref{eq:coverings-M-1}--\eqref{eq:coverings-M-2} onto the local coordinates $(x,y_u)$ induced by $\gamma_M$. The covering relations~\eqref{eq:coverings-1}--\eqref{eq:coverings-M-2} form the covering sequences that must satisfy cone conditions, as required for condition~\emph{(B2)}.

\subsubsection{Validating cone conditions when $d_x = 1$}
\label{sec:validation-cones}
To prove cone conditions for $d_x = 1$ we make use of the notion of an \emph{interval enclosure of a derivative}. This means that each element in the Jacobian matrix is an interval given by the range of possible values achieved by the corresponding partial derivatives. For $g : \mathbb{R}^n \to \mathbb{R}^n$ and a set $U \subset \mathbb{R}^n$ we write the interval enclosure of the $n \times n$ Jacobian matrix $Dg$ in $U$ as
\begin{displaymath}
  \left[ Dg(U) \right] := \left[ \left( Dg(U) \right)_{k,m} \right]_{k,m \in \{1, \ldots, n\}} \subset \mathbb{R}^{n\times n},
\end{displaymath}
where,
\begin{equation*}
  \left( Dg(U) \right)_{k,m} = \left[ \inf_{z \in U} \frac{\partial \pi_{x_k} g}{\partial x_m}(z),\,
    \sup_{z \in U} \frac{\partial \pi_{x_k} g}{\partial x_m}(z) \right].
\end{equation*}
For a vector $v \in \mathbb{R}^n$, the matrix product $\left[ Dg(U) \right] \, v$ is given by
\begin{displaymath}
  \left[ Dg(U) \right] \, v := \left\{ D \, v \bigm| D \in \left[ Dg(U) \right] \right\} \subset \mathbb{R}^n.
\end{displaymath}

We also introduce the following notation for a particular choice of cone. Recall from Equation~\eqref{eq:cone-def-2} that a cone in local coordinates at $z_{\gamma_i} = \gamma_i(z) \in \mathbb{R}^{d_x} \times \mathbb{R}^{d_y}$ is defined as
\begin{displaymath}
  {\cal C}_{\gamma_i}(z_{\gamma_i}) := \left\{  (q_x,\, q_y) \in \mathbb{R}^{d_x} \times \mathbb{R}^{d_y} \Bigm| \left\| q_x - \pi_x(z_{\gamma_i}) \right\|_{d_x}  > \left\| q_y - \pi_y(z_{\gamma_i}) \right\|_{d_y} \right\}.
\end{displaymath}
In our special case with $n = 3$ and local coordinates $(x, y_u, u_s)$, we choose the $L_\infty$-norm for both projections, that is,
\begin{displaymath}
  \left\| \pi_x(x, y_u, u_s) \right\|_{d_x} = \left\| x \right\|_1 = \left\vert x \right\vert,
\end{displaymath}
and
\begin{displaymath}
  \left\| \pi_y(x, y_u, u_s) \right\|_{d_y} =  \left\| y\right\|_2 = \left\| (y_u,\, y_s) \right\|
  := \max \left\{ \frac{\left\vert y_u\right\vert}{\kappa_u},\, \frac{\left\vert y_s\right\vert}{\kappa_s} \right\},
\end{displaymath}
for some positive constants $\kappa_u, \kappa_s > 0$. Using these norms, we define a cone at $z \in N_i$ as
\begin{displaymath}
  {\cal C}_i(z) = \gamma_i \left( {\cal C}_{\gamma_i}(z_{\gamma_i};\, \kappa_u, \kappa_s) \right),
\end{displaymath}
where
\begin{eqnarray*}
  {\cal C}_{\gamma_i}(z_{\gamma_i};\, \kappa_u, \kappa_s) = \big\{ q \in \mathbb{R}^n
  &\bigm|&  q - \gamma_i(z) = x \, (1, y_u, y_s) \\
  & & \text{with} \ \left\vert y_u\right\vert <\kappa_u, \left\vert y_s\right\vert <\kappa_s \text{ and } x\in \mathbb{R}  \big\}.
\end{eqnarray*}
With the above definition of a cone, we validate that the family of h-sets $\{ N_{abc} \}_{abc \in I}$, with $ab \in \left\{ 10, 11, 12, 13, 14, 20, 21, 22, 23 \right\}$ and $c \in \{ 0, \ldots, 49\}$ satisfies cone conditions; this can be done by using only local coordinates with the following lemma.


\begin{lemma}
\label{lem:cone-propagation}
Let $i, j \in I$, and let the h-sets $N_i$ and $N_j$ be equipped with the cones
\begin{eqnarray*}
  {\cal C}_i(z) &=& \gamma_i \left( {\cal C}_{\gamma_i}(z_{\gamma_i};\, \kappa_u^i, \kappa_s^i) \right) \text{ and} \\
  {\cal C}_j(z) &=& \gamma_j \left( {\cal C}_{\gamma_j}(z_{\gamma_j};\, \kappa_u^j, \kappa_s^j) \right),
\end{eqnarray*}
for fixed constants $\kappa_u^i, \kappa_s^i, \kappa_u^j, \kappa_s^j > 0$, respectively. Suppose that for every $v \in \{ 1 \} \times [-\kappa_u^i,\, \kappa_u^i] \times [-\kappa_s^i,\, \kappa_s^i]$ and for every vector
\begin{displaymath}
  w = (w_x, w_{y_u}, w_{y_s}) \in \left[ Df_{ji}(N_{\gamma_i}) \right] \, v
\end{displaymath}
we have $w_x \neq 0$ and
\begin{equation}
\label{eq:cone-prop}
  \left( 1,\, \frac{w_{y_u}}{w_x},\, \frac{w_{y_s}}{w_x} \right) \in \{ 1 \} \times [-\kappa_u^j,\, \kappa_u^j] \times
  [-\kappa_s^j,\, \kappa_s^j]
\end{equation}
Then $f$ satisfies cone conditions from $N_i$ to $N_j$.
\end{lemma}
\begin{proof}
Let $z \in N_i$ and take $q = (q_x,\, q_y) \in {\cal C}_{\gamma_i}(z_{\gamma_i};\, \kappa _u^i, \kappa _s^i)$. Since $\gamma_j(f(z)) = f_{ji}(z_{\gamma_i})$, we must show that $f_{ji}(q) \in {\cal C}_{\gamma_j}(\gamma_j(f(z));\, \kappa _u^j, \kappa _s^j)$. From the Mean Value Theorem, we have
\begin{displaymath}
  f_{ji}(q) - f_{ji}(z_{\gamma_i}) \in \left[ Df_{ji}(N_{\gamma_i}) \right]\, (q - z_{\gamma_i}),
\end{displaymath}
which implies that
\begin{eqnarray*}
  f_{ji}(q) - f_{ji}(z_{\gamma_i}) &=& \left\vert \pi_x(q - z_{\gamma_i}) \right\vert \, \left[ Df_{ji}(N_{\gamma_i}) \right]\, \frac{q - z_{\gamma_i}}{\left\vert \pi_x(q - z_{\gamma_i}) \right\vert} \\
  &=& \left\vert \pi_x(q - z_{\gamma_i}) \right\vert \, w,
\end{eqnarray*}
for some $w = (w_x, w_{y_u}, w_{y_s}) \in \left[ Df_{ji}(N_{\gamma _i}) \right] \, v$, with $v = (q - z_{\gamma_i}) / \left\vert \pi_x(q - z_{\gamma_i}) \right\vert$. Our choice of norms combined with assumption~\eqref{eq:cone-prop} guarantees that $w \in {\cal C}_{\gamma_j}(0;\, \kappa _u^j, \kappa _s^j)$. Hence, $f_{ji}(q) - f_{ji}(z_{\gamma_i}) \in {\cal C}_{\gamma_j}(0;\, \kappa _u^i, \kappa _s^i)$ and, thus, $f_{ji}(q) \in {\cal C}_{\gamma_j}(f_{ji}(z);\, \kappa_u^j, \kappa _s^j)$, as required. \newline
\end{proof}

We apply Lemma~\ref{lem:cone-propagation} as follows. Recall that the h-sets in our family are defined in local coordinates as
\begin{displaymath}
  \gamma_{ab}^{-1}(N_{abc}) := [-0.1, 0.1] \times \Big( \mathcal{I}_{abc} \times [-0.4, 0.4] \Big),
\end{displaymath}
and the intervals $\mathcal{I}_{abc}$ are defined recursively based on iterates of the intervals
\begin{displaymath}
  \mathcal{I}_{10c} := c \, \Delta + [-\Delta,\, \Delta]
  \qquad \mbox{and} \qquad
  \mathcal{I}_{20c} := -c \, \Delta + [-\Delta ,\, \Delta],
\end{displaymath}
with $\Delta = \frac{2}{50} = 0.04$. We choose $\kappa_u = \kappa_s = \frac{1}{2} \Delta = 0.02$ as the positive constants associated with cones for the h-sets $N_{10c}$ and $N_{20c}$ for $c = 0, \ldots, 49$. Hence, for $i \in \{ 10c, 20c \}_{c \in \{0, \ldots, 49\}}$, the cones are defined in local coordinates as
\begin{displaymath}
  \begin{array}{rcl}
    {\cal C}_{\gamma_i}(z_{\gamma_i};\, \frac{1}{2} \Delta, \frac{1}{2} \Delta)
    = \big\{ q \in \mathbb{R}^n \bigm| \frac{1}{2}\Delta \, \left\vert \pi_x \left( q - \gamma_i(z) \right) \right\vert &>& \left\vert \pi_{y_u}\left( q - \gamma_i(z) \right) \right\vert \text{ and } \smallskip \\
   \frac{1}{2} \Delta \, \left\vert \pi_x \left( q - \gamma_i(z) \right) \right\vert
    &>& \left\vert \pi_{y_s}\left( q - \gamma_i(z) \right) \right\vert \big\} .
  \end{array}
\end{displaymath}
We note that, due to our choice $\kappa_u = \kappa_s = \frac{1}{2} \Delta$ and the fact that the sets $\{ N_{10c}, $ $N_{20c} \}_{c = 0, \ldots, 49}$ overlap and \eqref{eq:M-parts}, for every horizontal disk $[h]$ in $M$ we can choose some $N_i \in \{ N_{10c}, N_{20c} \}_{c = 0, \ldots, 49}$ so that $[h] \cap N_i = [h_i]$, for some horizontal disk $[h_i]$ in $N_i$. This ensures \emph{(B1)}.
We propagate the cones by using Lemma~\ref{lem:cone-propagation} with appropriate choices for the constants $\kappa_u$ and $\kappa_s$ to ensure that cone conditions are satisfied for each of the coverings. In this computer-assisted way, we show that condition~\emph{(B2)} is satisfied.

\subsection{Final steps in proving Theorem~\ref{th:henon}}
The validation of the conditions \emph{(B1)} and \emph{(B2)} described in Section~\ref{sec:validation} for the family of h-sets $\{ N_i \}_{i \in I}$ constructed in Section~\ref{sec:boxes} is performed for any choice $\xi \in [ 1.01,\, 1.125]$ for the H{\'e}non-like family~\eqref{eq:Henon-family}. The proof that there exists a set $U\subset \mathbb{R}^3$ such that $\bigcup N_i\subset U$ and the invariant set of $f$ in $U$ is hyperbolic and transitive uses standard, earlier established methods and can be found in the Appendix. Since the sets involved in our construction depend on the choice of the parameter $\xi$, we apply a form of interval enclosure on the parameter as well. More precisely, we subdivide the parameter interval $[1.01,1.125]$ into $115$ sub-intervals of width $10^{-3}$ and conduct our construction separately on each of these sub-intervals; the computer-assisted validation over the entire range of $\xi$-values takes under a second on a standard laptop. Taken together, we obtain a proof of the existence of a $2$-blender for the $\xi$-family of H{\'e}non-like maps with $\xi \in [1.01,\, 1.125]$.

We have used the CAPD library\footnote{Computer Assisted Proofs in Dynamics: http://capd.ii.uj.edu.pl.} \cite{CAPD2021} as the tool for our interval arithmetic computations.

\section{Discussion and conclusions}
\label{sec:conclusions}

We presented a computer-assisted method for proving the existence of a blender for a given family of diffeomorphisms. The algorithm requires construction of a finite number of covering sequences with h-sets that intersect a transitive hyperbolic set. We select the necessary covering sequences by following different homoclinic excursions, but this is only a convenient approach and the algorithm itself does not depend on knowledge of homoclinic or heteroclinic connecting orbits. The method is very flexible in that it can be applied to ranges of a specified parameter. We successfully applied the algorithm in Section~\ref{sec:CAP} to demonstrate existence of a $2$-blender for the H{\'e}non-like family~\eqref{eq:Henon-family} from~\cite{hkos_blenderDEA, hkos_blender} over the range $\xi \in [1.01,\, 1.125]$ of the weak expansion rate $\xi$. Our construction involves two different homoclinic orbits and, notably, we chose not to include the two fixed points of the H{\'e}non-like family in the family of h-sets. Hence, the blender we identified here is an invariant subset of the (maximal) blender that has been considered in previous work~\cite{hkos_blender, hkos_blenderDEA, hkos_boxedin}. Indeed, our construction and computer-assisted method of proof of existence of a blender do not require first to identify the underlying  \emph{maximal} transitive hyperbolic set.

Our objective here was to provide a proof of concept for an example previously considered in the literature, and we have been able to do so without difficulty and in a computationally efficient manner. Numerical evidence in \cite{hkos_blenderDEA} suggests that the blender exists well beyond the range proven with our algorithm (up to $\xi \approx 1.6$). We believe that the $\xi$-range of validity of our computer-assisted proof could be extended via a combination of the following three considerations.
\begin{itemize}
\item
When using the h-sets introduced above, the bottle neck for the proof with $\xi >1.125$ is that the final coverings $N_{14c} \cover{f} M$ for $c = 0,\ldots ,49$ fail, because the sets do not align in the local weakly expanding coordinate $y_u$. This is due to the large expansion by $\xi^5$ that results from taking five steps along a sequence of coverings. A remedy could be to consider larger h-sets, which require fewer steps in the sequences of coverings. However, this would complicate validation of the required conditions. For example, in our current setting, to validate covering and cone conditions we compute interval enclosures of $Df_{ji}(N_{\gamma_i})$, $f_{ji}(N_{\gamma_i})$, $f_{ji}(N^-_{\gamma_i,l})$ and $f_{ji}(N^-_{\gamma_i,r})$ without any subdivisions of the sets $N_{\gamma_i}$, $N^-_{\gamma_i,l}$ and $N^-_{\gamma_i,r}$, and our bounds turn out to be sharp enough. Enlarging the h-sets would likely require subdivisions, which would slow down the computations.
\item
We are only using the homoclinic connections associated with the hyperbolic fixed point $p^+$. One could perform a similar construction by using $p^-$, which could lead to a different, extended parameter range for $\xi$.
\item
We could place one mother h-set $M^+$ close to $p^+$ and another mother h-set $M^-$ close to $p^-$ and, instead of using homoclinic orbits, position the h-sets along heteroclinic orbits between $p^+$ and $p^-$. Such heteroclinic connections will require fewer iterates.
\end{itemize}
It is beyond the scope of this paper to consider further refinements with such more complicated constructions of covering relations and to obtain a more comprehensive coverage of the known $\xi$-range with a blender of the H{\'e}non-like family. Again, our goal was to introduce and demonstrate a general method. Indeed, proving the existence of blenders, and also of heterodimensional cycles, in other dynamical systems remains an interesting and challenging task for future research.

\section*{Acknowledgements}
Authors would like to express sincere thanks to the anonymous reviewer for his corrections and suggestions, which allowed us to substantially improve the earlier version of this paper.

B.K.\ and H.M.O.\ are grateful for the hospitality and financial support during their one-week visit in Krak{\'o}w. M.C.\ was partially supported by NCN grants 2019/35/B/ST1 /00655 and 2021/41/B/ST1/00407. The research of B.K.\ and H.M.O.\ was partially supported by Royal Society Te Ap\={a}rangi Marsden Fund grant \#22-UOA-204. P.Z.\ was partially supported by NCN grant 2019/35/B/ST1/00655.

\appendix
\section{Establishing hyperbolicity and transitivity}
\label{sec:hyp-trans}
In Section~\ref{sec:CAP}, we applied Theorem~\ref{th:main} to prove the existence of a $2$-blender for the H\'enon-like family~\eqref{eq:Henon-family}. For completeness, we explain how we validate with computer-assisted tools that there exists a set $U$ whose invariant set is hyperbolic and transitive, and which contains $\bigcup_{i \in I} N_i$, as required for Theorem~\ref{th:main}.

Consider a set of indexes $J:=\{0,11,12,13,14,21,22,23\}$ and assume that we have a sequence of h-sets $L_j$ for $j\in J$. These h-sets will have $d_x=2$ and $d_y=1$. This means that the dimensions of the exit set matches the dimension of the unstable bundle, and the dimension of the entry set matches the dimension of the stable bundle. We will choose these sets to ensure that
\begin{align}
  	&L_{0} \cover{f} L_{11} \cover{f} L_{12} \cover{f} L_{13}\cover{f} L_{14} \cover{f} L_{0} , \label{eq:L-seq-1}\\
	&L_{0} \cover{f} L_{21} \cover{f} L_{22} \cover{f} L_{23} \cover{f} L_{0}.\label{eq:L-seq-2}
\end{align}
We emphasise that in \eqref{eq:L-seq-1}--\eqref{eq:L-seq-2} the topological alignment for the coverings agrees with the contraction and expansion of the system. The topological alignment involved in the coverings from the sequences \eqref{eq:L-seq-1}--\eqref{eq:L-seq-2} is depicted in Figure \ref{fig:covering-2}. Due to the good topological alignment of the exit and entry coordinates with the coordinates of hyperbolic contraction and expansion we can ensure that the sequences of coverings  \eqref{eq:L-seq-1}--\eqref{eq:L-seq-2} start and finish with the same set $L_{0}$.

For our h-sets $L_j$ we use the same local coordinates as for the h-sets $N_j$. Namely, we use $\gamma_M$ defined in \eqref{eq:coordinate-change-M} for the local coordinates of $L_{0}$ and $\gamma_{1b},\gamma_{2b}$ defined in \eqref{eq:coordinate-changes-i} for the local coordinates of $L_{j}$ for $j\in J\setminus \{0\}$. (The local coordinates $(x_1,x_2,y)$ of the sets $L_{0}$, $L_{1b}$ and $L_{2b}$ correspond to the local coordinates $(x,y_u,y_s)$ of the sets $M$, $N_{1b}$ and $N_{2b}$, respectively; namely, $x_1=x$, $x_2=y_u$ and $y=y_s$.)

\begin{figure}[t!]
  \centering
  \includegraphics{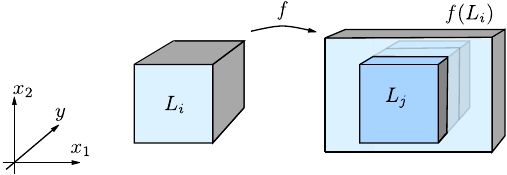}
  \caption{\label{fig:covering-2}
    Illustration of a covering relation $L_{i} \cover{f} L_{j}$ with $d_u=2$ and $d_s=1$.}
\end{figure}

We have chosen our sets so that
\begin{align}
	M \subset & \, L_{0} \notag \\
	N_{1bc}\subset & \, L_{1b} \qquad \mbox{for } c \in \{0,...,49\}\mbox{ and } b=1,2,3,4, \label{eq:N-sub-L} \\
	N_{2bc}\subset & \, L_{2b} \qquad \mbox{for } c \in \{0,...,49\}\mbox{ and } b=1,2,3. \notag
\end{align}
Because of \eqref{eq:N-sub-L} we have
\[
	\bigcup_{i\in I} N_i \subset U:= \bigcup_{j\in J} L_{j}.
\]
In more detail, we have chosen the sets $L_{0},L_{1b}$ and $L_{2b}$ so that their projections onto the local coordinates $x_1,y$ are the same as for $M$, $N_{1b}$ and $N_{2b}$, respectively. (This also means that the projections of $L_{0},L_{1b}$ and $L_{2b}$ onto the coordinates $X,Y$ are the same as the projections of $M$, $N_{1b}$ and $N_{2b}$, respectively.) Moreover, we have chosen the sets $L_{0}$, $L_{1b}$ and $L_{2b}$ to be larger along the local coordinate $x_2$ than the sets $M$, $N_{1b}$ and $N_{2b}$, respectively.
In our computer program we have chosen the set $L_{0}$ in the local coordinates given by $\gamma_M$ to be of size $\pi_{x_2}L_{\gamma_M}=[-100,100]$. This means that $M\subset L_{0}$ since $\pi_{x_2}L_{\gamma_M}$ is much larger than $\pi_{x_2}M=[-2,2]$; see \eqref{eq:M-set}. We have chosen such a large interval since it works well for all parameters $\xi \in [1.01,\, 1.125]$. We choose the sizes of the intermediate sets in the covering sequences \eqref{eq:N-sub-L} so that for a covering $L_{i} \cover{f} L_{j}$ we have $\pi_{x_2}f_{ji}(L_{\gamma_i}) \subset \pi_{x_2}L_{\gamma_j}$ and so that $N_{jc}\subset L_j$ for $c=0,\ldots,49$ and $j\in J\setminus \{0\}$; this is illustrated in Figure \ref{fig:covering-2}. (This choice is performed automatically by our computer program.) We then
use an analogue of Lemma \ref{lem:cov-1d-unstble} to validate the coverings \eqref{eq:L-seq-1}--\eqref{eq:L-seq-2}.

\subsection{Establishing hyperbolicity}
\label{sec:validation-hyper}
To validate hyperbolicity we can use the notion from~\cite{MR2736320} of a strongly hyperbolic covering relation. Recall that our h-sets $\{L_j\}_{j\in J}$ are equipped with the local coordinates $x=(x_1,x_2) \in \mathbb{R}^{d_x}=\mathbb{R}^2$ and $y  \in \mathbb{R}^{d_y}=\mathbb{R}$. The coordinates $(x_1, x_2)$ are $x_x=d_u=2$ dimensional and expanding, and $y$ is $d_y=d_s=1$ dimensional and contracting. However, even though $d_u + d_s = n$, we do not assume or require that the coordinates $(x_1,x_2)$ and $y$ are perfectly aligned with the unstable and stable fibre bundles $E^u$ and $E^s$, respectively.

\begin{definition}[Strong hyperbolicity of h-sets~\cite{MR2736320}]
\label{def:strong-hyperbolicity}
Let $\{ L_j \}_{j \in J}$ be a family of h-sets associated with diffeomorphisms $f_{ji} : \mathbb{R}^{d_x} \times \mathbb{R}^{d_y} \to \mathbb{R}^n$. Define the $n \times n$ diagonal matrix
\begin{equation}
  Q = \mathrm{diag}\left( \mathrm{Id}_{d_u},\, -\mathrm{Id}_{d_s} \right), \label{eq:Q-def}
\end{equation}
where $\mathrm{Id}_{d_u}$ and $\mathrm{Id}_{d_s}$ are the identity matrices of dimensions $d_u$ and $d_s$, respectively.  We say that $\bigcup_{j \in J} L_j$ is \emph{strongly hyperbolic} if for every $i, j \in J$ such that $f(L_i) \cap L_j \neq \emptyset$, the matrix
\begin{equation}
\label{eq:strong-hyperbolicity}
  \left[ Df_{ji}(\gamma_i(z)) \right]^\top  Q \, \left[ Df_{ji}(\gamma_i(z)) \right] - Q,
\end{equation}
is positive definite for every $z \in L_i \cap f^{-1}(L_j)$.
\end{definition}

To establish that an invariant set in $U=\bigcup_{j \in J}L_j$ is hyperbolic we can now invoke a known result from~\cite{MR2736320}, namely,

\begin{theorem}\cite{MR2736320}
\label{th:hyperbolic-set}
If $U=\bigcup L_j$ is strongly hyperbolic then the (forward and backward) invariant set
\begin{displaymath}
  \Lambda = \mathrm{Inv}\left( f,\, U \right)
\end{displaymath}
is a hyperbolic set.
\end{theorem}

For every local map $f_{ji}$ involved in the coverings \eqref{eq:L-seq-1}--\eqref{eq:L-seq-2}, we follow Definition \ref{def:strong-hyperbolicity} and validate that every matrix from the interval enclosure
\begin{equation}
\label{eq:strong-hyp-cond-cap}
  \left[ Df_{ji}(L_{\gamma_i}) \right]^\top Q \, \left[ Df_{ji}(L_{\gamma_i}) \right] - Q
\end{equation}
is positive definite; this proves hyperbolicity of the invariant set that is contained in $U$ for the H{\'e}non-like family~\eqref{eq:Henon-family}. Theorem~\ref{th:hyperbolic-set} does not guarantee that $\Lambda $ is non empty. However, this is ensured by conditions~\emph{(B1)} and~\emph{(B2)}, as pointed out in Remark \ref{rem:inv-non-empty}.

\subsection{Establishing transitivity}
To establish transitivity we show that the invariant set $\Lambda$ in $U$ for $f$ is conjugate to a transitive set $\Lambda_\Sigma \subset \Sigma = \{0,1\}^{\mathbb{Z}}$ for the shift map.
For this, we use the following proposition. \newline

\begin{proposition}
\label{prop:conjugacy}
Consider an infinite sequence of covering relations
\begin{displaymath}
  \cdots \cover{f} L_{j_{-1}} \cover{f} L_{j_0} \cover{f} L_{j_1} \cover{f} L_{j_2} \cover{f} \cdots,
\end{displaymath}
where $L_{j_k}$ is selected from the finite family of h-sets $\{ L_j \}_{j \in J}$ for all $k \in \mathbb{Z}$. For each covering $L_{j_k} \cover{f} L_{j_{k+1}}$, define the interval enclosure $[ Df_k ( L_{\gamma_{j_k}} )]$, where $f_k := \gamma_{j_{k+1}} \circ f \circ \gamma_{j_k}$ represents $f$ in local coordinates. Suppose for every $f_k$, we have that every $n \times n$ matrix in the interval enclosure
\begin{equation}
\label{eq:conjugacy-condition}
  [ Df_k( L_{\gamma_{j_k}} ) ]^\top Q \, [ Df_k (L_{ \gamma_{j_k}} )]- Q
\end{equation}
is strictly positive definite; here, the matrix $Q$ is the same matrix as defined in~\eqref{eq:Q-def}. Then, for every $\varepsilon > 0$ there exists $m = m(\varepsilon) > 0$, such that for $z_1, z_2 \in L_{j_0}$ with
\begin{displaymath}
  f^k(z_1), f^k(z_2) \in L_{j_k} \qquad\text{for all } k \in \{ -m, \ldots, m \},
\end{displaymath}
we have $\| z - z^\prime \| < \varepsilon$. (Here we use the standard Euclidean norm $\|\cdot \|$ in $\mathbb{R}^n$.)
\end{proposition}
\begin{proof}
Let $z_1, z_2 \in L_{j_0}$ and define $z_{\ell}^\gamma := \gamma_{j_0}(z_{\ell})$ for ${\ell} = 1, 2$. We will consider two cases, namely, $\left( z_1^\gamma - z_2^\gamma \right)^\top Q \, \left( z_1^\gamma - z_2^\gamma \right) \geq 0$ and $\left( z_1^\gamma - z_2^\gamma \right)^\top Q \, \left( z_1^\gamma - z_2^\gamma \right) < 0$.

Assume first that $\left( z_1^\gamma - z_2^\gamma \right)^\top Q \, \left( z_1^\gamma - z_2^\gamma \right) \geq 0$. Since the index set $J$ is finite and the h-sets $L_j$, for $j \in J$, are compact, the family of interval enclosures~\eqref{eq:conjugacy-condition} for all $k \in \mathbb{Z}$ is compact. This means that for suitably small $\delta > 0$ and for $\lambda > 1$ suitably close to $1$, every matrix in the interval enclosure
\begin{displaymath}
 [ Df_k( L_{\gamma_{j_k}} ) ]^\top Q \, [ Df_k (L_{ \gamma_{j_k}} )]  - \lambda \,  Q - \delta \, \mathrm{Id}_n
\end{displaymath}
is strictly positive definite. Let us introduce the auxiliary notation $q(z) := z^\top Q z$. Taking $z = (x, y)=(x_1,x_2, y)$ from~\eqref{eq:Q-def}, we observe that
\begin{displaymath}
  q(z) = q(x, y) = (x, y)^\top Q \, (x, y) = \| x \|^2 - \| y \|^2 \leq \| x \|^2 + \| y \|^2 = \| z \|^2.
\end{displaymath}
By the Mean Value Theorem, for any $p_1, p_2 \in L_{\gamma_{j_k}}$, we have
\begin{align*}
  f_k(p_1) - f_k(p_2) & = \int_0^1 \frac{d}{ds} f_k\left( p_2 + s \, (p_1 - p_2)  \right) \; d s \\
                      & = \left( \int_0^1 Df_k\left( p_2 + s \, (p_1 - p_2) \right) \; d s \right) \, (p_1 - p_2) \\
                      & =: D \, (p_1 - p_2),
\end{align*}
where $D \in [ Df_k ( L_{\gamma_{j_k}} ) ]$ depends on $p_1$ and $p_2$. Consequently,
\begin{align*}
  & q \left( f_k(p_1) - f_k(p_1) \right) - \lambda \, q(p_1 - p_2) - \delta \, \| p_1 - p_2 \|^2 \\
  & = (p_1 - p_2)^\top \left( D^\top Q \, D - \lambda \, Q - \delta \, \mathrm{Id}_n \right) \, (p_1 - p_2) > 0.
\end{align*}
This means that, since $q(z_1^\gamma - z_2^\gamma) \geq 0$,
\begin{align*}
  & \| f_m \circ \cdots \circ f_0(z_1^\gamma) - f_m \circ \cdots \circ f_0(z_2^\gamma) \|^2 \\
  & \geq q \left( f_m \circ \cdots \circ f_0(z_1^\gamma) - f_m \circ \cdots \circ f_0(z_2^\gamma) \right) \\
  & \geq \lambda \, q \left(  f_{m-1} \circ \cdots \circ f_0(z_1^\gamma) - f_{m-1} \circ \cdots \circ f_0(z_2^\gamma) \right) \\
  & \vdots \\
  & \geq \lambda^m \, q \left( f_0(z_1^\gamma) - f_0(z_2^\gamma) \right) \\
  & \geq \lambda^m \left( \lambda \, q(z_1^\gamma - z_2^\gamma) + \delta \, \| z_1^\gamma - z_2^\gamma \|^2 \right)  \\
  & \geq \lambda^m \, \delta \, \| z_1^\gamma - z_2^\gamma \|^2.
\end{align*}
Since $\lambda > 1$ and $\delta > 0$, we will have $\| z_1^\gamma - z_2^\gamma \| < \varepsilon$ as required, provided we choose $m$ sufficiently large.

The proof for the case $\left( z_1^\gamma - z_2^\gamma \right)^\top Q \, \left( z_1^\gamma - z_2^\gamma \right)  < 0$ follows from mirror arguments, using the inverse map. \newline
\end{proof}

We are now ready to prove that the invariant set $\Lambda$ in $U=\bigcup L_j$ for the $\xi$-dependent H\'{e}non-like family~\eqref{eq:Henon-family} with $\mu = -9.5$ and $\beta = 0.3$ is transitive. Let us equip $\Sigma = \{  0, 1 \}^{\mathbb{Z}}$ with the usual metric $d : \Sigma \times \Sigma \to \mathbb{R}_+$, where
\begin{displaymath}
  d(a, b) = \sum_{k \in \mathbb{Z}} \frac{1}{2^{|k|}} \left\vert a_k - b_k \right\vert,
\end{displaymath}
and consider the usual shift map $\sigma : \Sigma \rightarrow \Sigma$, with $\sigma\left( (a_k)_{k \in \mathbb{Z}} \right) = (a_{k-1})_{k \in \mathbb{Z}}$. Let us consider $\alpha = 1\, 0\, 0\, 0\, 0$ and $\beta = 1\, 0\, 0\, 0$, and the set $\Lambda_\Sigma$ that consists of (infinite) sequences in $\Sigma$ constructed from the finite sequences $\alpha$ and $\beta$, and from shifts of such sequences. Note that $\sigma(\Lambda_\Sigma) = \Lambda_\Sigma$ and $\Lambda_\Sigma$ is transitive for the map $\sigma$. We show that $\Lambda$ is transitive by constructing a conjugacy between the dynamics on $\Lambda$ and on $\Lambda_\Sigma$.

Consider the transformation $\phi : \Lambda \to \Lambda_\Sigma$ that assigns to a point $z \in \Lambda$ a sequence $a = \phi(z) \in \Lambda_\Sigma$ according to the following rule
\begin{displaymath}
  a_k = 1 \quad \Longleftrightarrow \quad f^k(z) \in L_{0}.
\end{displaymath}

The intuition behind this definition of $\phi$ is based on the two covering sequences \eqref{eq:L-seq-1} and \eqref{eq:L-seq-2}. If $z \in L_{0} \cap \Lambda$ then the trajectory starting from $z$ will pass through either the first or the second sequence of coverings \eqref{eq:L-seq-1} or \eqref{eq:L-seq-2}, respectively. Hence, we either have $\alpha$ or $\beta$ as the first entries in $\phi(z)$. The remaining entries are also determined by the order in which the trajectory follows these two different sequences of coverings.

The two sequences of coverings \eqref{eq:L-seq-1} and \eqref{eq:L-seq-2}  can be glued in arbitrary order to produce infinite sequences, and through any such infinite sequence there exists an orbit of $f$ passing through it (see \cite{MR2060531}).
Hence, it follows that $\phi$ is surjective, and we have the conjugacy
\begin{displaymath}
  \phi \circ f = \sigma \circ \phi.
\end{displaymath}
Moreover, continuity of $f$ implies continuity of $\phi$ (if we choose two sufficiently close initial points in $\Lambda$, their trajectories pass through the same sequences of coverings and, hence, follow the same symbols, for a prescribed length). Continuity of $\phi^{-1}$ follows from Proposition~\ref{prop:conjugacy}. Therefore, there is a conjugacy between $\Lambda$ and $\Lambda_\Sigma$, and the fact that $\Lambda_\Sigma$ is transitive implies that $\Lambda$ is transitive.

We finish by noting that the interval enclosure defined in \eqref{eq:conjugacy-condition} is the same as the one from \eqref{eq:strong-hyp-cond-cap} that we used to validate  hyperbolicity.



\end{document}